\documentclass{article}
\usepackage{amsmath}
\usepackage{amsfonts}
\usepackage{amsthm}
\usepackage{enumerate}
\usepackage{amssymb}

\newtheorem{thm}{Theorem}[section]
\newtheorem{prop}[thm]{Proposition}
\newtheorem{cor}[thm]{Corollary}
\newtheorem{lem}[thm]{Lemma}

\newtheorem*{thmnn}{Theorem}

\theoremstyle{remark}
\newtheorem{rmk}[thm]{Remark}

\theoremstyle{definition}
\newtheorem{defn}[thm]{Definition}

\begin{document}

\title{Shimura Lifts of Weakly Holomorphic Modular Forms}

\author{Yingkun Li \thanks{Partially supported by DFG grant BR-2163/4-1.} , Shaul Zemel\thanks{This work was carried out mainly while the author was working at the Technische Universit\"{a}t Darmstadt in Germany, and was partially supported by DFG grant BR-2163/4-1.}}

\maketitle


\section*{Introduction}

The systematic investigation of half-integral weight modular forms began with Shimura's seminal paper \cite{[Sh]}, where he developed the Hecke theory for these forms. To each eigenform of weight $k+\frac{1}{2}$, level $4N$, and character $\chi$, Shimura attached an eigenform of weight $2k$ and character $\chi^{2}$, with some level. This is known as the \emph{Shimura lift}. Niwa proved in \cite{[N]} that one can always take $2N$ to be the level of the lift, as conjectured by Shimura.

Later, Kohnen defined in \cite{[K1]} and \cite{[K2]} a certain plus-space of half-integral weight modular forms of level $4N$ with $N$ odd and square-free. He then showed that the Shimura lifts of eigenforms in these plus-spaces have level $N$. The paper \cite{[UY]} extended the results of \cite{[K1]} and \cite{[K2]} to the case where $N$ may be even and square-free. In another direction, \cite{[C1]} gave explicit formulae for Shimura lifts of products of modular forms of integral weight and certain theta functions. This was followed, and generalized, by \cite{[C2]}, \cite{[HN]} and others.

Shintani gave in \cite{[Sn]} a construction, which is reciprocal to that of Shimura, using Siegel's indefinite theta series. The paper \cite{[N]} of Niwa mentioned above also uses a theta function, which depends on the level and character of the input form. Going further in the direction of the theta lift, Borcherds investigated in \cite{[B]} a theta lift from elliptic modular forms to automorphic forms on orthogonal groups, a special case of the latter gives the classical elliptic modular forms once more. Borcherds' work systematically uses vector-valued modular forms, and employs the regularization technique in \cite{[HM]} to lift weakly holomorphic modular forms. The resulting modular forms are meromorphic with known poles. One of the examples presented in Section 14 of \cite{[B]} is referred to as a ``Shimura lift''.

\smallskip

In this paper, we aim to deduce Shimura's result from Borcherds' theta lift. The main consequence of this argument is an immediate extension of the Shimura lift to include \emph{weakly holomorphic} modular forms of half-integral weight, whose images will be meromorphic modular forms of even weights with poles at CM points. Our construction differs from the one of Niwa in \cite{[N]} in that \cite{[B]} integrates \emph{vector-valued} modular form against a vector-valued theta kernel. To obtain Shimura's classical result, which takes scalar-valued modular form as input, we first need to lift a scalar-valued modular form on a congruence subgroup to an appropriate vector-valued modular form on the full metaplectic group. We accomplish this in two steps. First we establish an isomorphism, which is interesting in its own right, between the plus-space of scalar-valued modular forms of level $4N$ and a space of vector-valued modular forms of level $N$ with a 2-dimensional Weil representation. Then we will generalize a lemma from \cite{[B]} to obtain vector-valued modular forms for the full metaplectic group. These vector-valued modular forms will transform with the Weil representation associated to a lattice closely related to the congruence group $\Gamma_{1}(N)$. At the final stage we verify that the construction from \cite{[B]} takes the resulting vector-valued modular form of level 1 to the Shimura lift of index 1 of the original scalar-valued modular form. The Shimura lift of general index is extracted out of the index 1 case using an adaptation of an argument from \cite{[N]}. We conclude with some finer analysis of these results to prove that if $4|N$ then the Shimura lift of \emph{any} modular form, not necessarily in the the plus-space, has level $N$.

A simplified version of our main result is as follows.
\begin{thmnn}
Let $f$ be a holomorphic modular form in $\mathcal{M}_{k+\frac{1}{2}}(\widetilde{\Gamma}_{0}(4N),\omega_{\chi})$, where $\chi$ is a character modulo $N$ (see Definition \ref{autform} and Equation \eqref{omegachi}). Suppose $f$ has Fourier expansion $\sum_{n=0}^{\infty}c_{n}q^{n}$ at $\infty$. Let $t$ be a square-free integer.
If $4|N$, or if $t$ is odd and $f$ lies in the appropriate plus-space, then the function \[\mathcal{S}_{t}^{(N)}f(\tau):=-\frac{L(1-k,\eta)}{2}c_{0}+\sum_{l=1}^{\infty}\sum_{\substack{d|l \\ (d,Nt)=1}}d^{k-1}\eta(d)c_{\frac{tl^{2}}{d^{2}}}q^{l}\] is a holomorphic modular form of weight $2k$ on $\Gamma_{0}(N)$ with character $\chi^{2}$. Here $\eta$ is a slight variation of $\chi$, which depends on $t$ (see Theorem \ref{Stf}). Otherwise, the level of modularity is $2N$. When $f$ is weakly holomorphic, the previous result still holds, with the same holomorphic Fourier expansion at the cusp. The resulting function will have, however, poles of order $k$ at CM points.
\end{thmnn}
In fact, if $k=1$ then the same result holds also for harmonic weak Maass forms whose images under the operator $\xi_{3/2}$ of \cite{[BF]} are cusp forms of weight $\frac{1}{2}$.

Another feature of this construction is that it is independent of how our lifted modular forms transforms under (the metaplectic inverse image of) $\Gamma_{0}(4N)$. Only modularity with respect to a group slightly larger than $\Gamma_{1}(4N)$ is required. The somewhat peculiar phenomenon of ``squaring of the characters'' is then explained in terms of certain automorphisms of the discriminant group of the lattice we work with.

\smallskip

This paper is divided into 6 sections. Section \ref{Wt1/2Z} introduces modular forms of half-integral weights using only characters instead of multiplier systems, to facilitate the applications in this paper. In Section \ref{2dLift} we present the plus-spaces and their isomorphisms with spaces of modular forms with 2-dimensional representations. Section \ref{2N^2Lift} investigates the lifts to modular forms of level 1 with representation $\rho_{1}(N)$ of dimension $2N^{2}$, and some of their automorphisms. Section \ref{L1N} describes a certain lattice, whose discriminant kernel is roughly $\Gamma_{1}(N)$ and the associated Weil representation is $\rho_{1}(N)$. In Section \ref{ThetaLift} we apply the theta lift from \cite{[B]}, and show that it is related to the Shimura lift of index 1. Finally, Section \ref{HiLev} considers the Shimura lift with general index, and determines the level of the Shimura lift depending on some parameters.

\smallskip

We are grateful to J. Bruinier for several interesting discussions of this topic.

\section{Automorphic Forms of Half-Integral Weight \label{Wt1/2Z}}

Let $\mathcal{H}=\big\{\tau\in\mathbb{C}\big|\mathrm{Im}\tau>0\big\}$ be the upper half-plane, and $Mp_{2}(\mathbb{R})$ denotes the double cover of
$SL_{2}(\mathbb{R})$. The elements of $Mp_{2}(\mathbb{R})$ are pairs $(\alpha,\phi)$ where $\alpha \in SL_{2}(\mathbb{R})$ and $\phi:\mathcal{H}\to\mathbb{C}$ is a holomorphic function satisfying $\phi^{2}(\tau)=j(\alpha,\tau)$. Here and throughout, $j(\alpha,\tau)$ is $c\tau+d$ if $\alpha=\big(\begin{smallmatrix} a & b \\ c & d\end{smallmatrix}\big)$. This is the usual factor of automorphy for the action of such an element $\alpha$ of $SL_{2}(\mathbb{R})$ (as well as of $Mp_{2}(\mathbb{R})$) on $\mathcal{H}$ via the fractional linear transformation $\alpha\tau:=\frac{a\tau+b}{c\tau+d}$. The multiplication rule in $Mp_{2}(\mathbb{R})$ is defined by \[\big(\alpha,\phi(\tau)\big)\big(\beta,\varphi(\tau)\big)=\big(\alpha\beta,\phi(\beta\tau)\varphi(\tau)\big).\]
The subgroup $Mp_{2}(\mathbb{Z}) \subseteq Mp_{2}(\mathbb{R})$ lying over $SL_{2}(\mathbb{Z})$ is the double cover of the latter group. It is generated by \[T:=\Bigg(\begin{pmatrix} 1 & 1 \\ 0 & 1\end{pmatrix},1\Bigg) \quad\mathrm{and}\quad S:=\Bigg(\begin{pmatrix} 0 & -1 \\ 1 & 0\end{pmatrix},\sqrt{\tau}\Bigg)\] with the relation $S^{2}=(ST)^{3}$. We also set
\begin{equation}
R:=ST^{-1}S^{-1}=\Bigg(\begin{pmatrix} 1 & 0 \\ 1 & 1\end{pmatrix},\sqrt{\tau+1}\Bigg)\quad\mathrm{and}\quad Z:=S^{2}=(ST)^{3}=(-I,i). \label{RZdef}
\end{equation}
The latter element has order 4 in $Mp_{2}(\mathbb{Z})$, and it generates the center of that group. We take the branch cut $\mathbb{C}\backslash(-\infty,0)$ for the square root function with $\sqrt{-x}=\sqrt{x} \cdot i\in\mathcal{H}$ for every positive real number $x$.

We shall work here with several types of congruence subgroups of $SL_{2}(\mathbb{Z})$ and $Mp_{2}(\mathbb{Z})$. Given a natural number $N$, let $\Gamma_{1}(N)\subseteq\Gamma_{0}(N) \subseteq SL_{2}(\mathbb{Z})$ be the standard congruence subgroups (see, e.g., page 13 of \cite{[DS]}), and $\Gamma^{0}(N):=\big(\begin{smallmatrix} N & 0 \\ 0 & 1\end{smallmatrix}\big)\Gamma_{0}(N)\big(\begin{smallmatrix} 1/N & 0 \\ 0 & 1\end{smallmatrix}\big)$. Their inverse images in $Mp_{2}(\mathbb{Z})$ will be denoted by $\widetilde{\Gamma}_{0}(N)$, $\widetilde{\Gamma}_{1}(N)$, and $\widetilde{\Gamma}^{0}(N)$ respectively. Other groups that will arise naturally in our investigations are
\begin{equation}
\Gamma_{N}:=\Gamma_{0}(4N)\cap\Gamma_{1}(N)\quad\mathrm{and}\quad\widetilde{\Gamma}_{N}:=
\widetilde{\Gamma}_{0}(4N)\cap\widetilde{\Gamma}_{1}(N)=\big\{(\alpha,\phi) \in Mp_{2}(\mathbb{Z})\mid\alpha\in\Gamma_{N}\big\}. \label{GammaN}
\end{equation}

\smallskip

An \emph{even lattice} is a free $\mathbb{Z}$-module $L$ of finite rank which
is endowed with a non-degenerate $\mathbb{Z}$-valued bilinear form $(\cdot,\cdot)$, such that $\lambda^{2}:=(\lambda,\lambda)$ is even for every $\lambda \in L$. Equivalently, $L$ admits a $\mathbb{Z}$-valued quadratic form $Q$, which is related to the bilinear form via the equality $Q(\lambda)=\frac{\lambda^{2}}{2}$ and the polarization identity
$(\lambda,\mu)=Q(\lambda+\mu)-Q(\lambda)-Q(\mu)$ for all $\lambda$ and $\mu$ in $L$. The dual lattice $\mathrm{Hom}(L,\mathbb{Z})$ is then identified, through the bilinear form $(\cdot,\cdot)$, with the lattice \[L^{*}:=\big\{\lambda \in L_{\mathbb{R}}\big|(\lambda,\mu)\in\mathbb{Z}\mathrm{\ for\ all' }\mu \in L\big\} \subseteq L_{\mathbb{R}}:=L\otimes\mathbb{R},\] which contains $L$. The quotient group $D_{L}:=L^{*}/L$, called the \emph{discriminant group} of $L$, has finite cardinality $\Delta_{L}$. The quadratic form $Q$ induces a well-defined $\mathbb{Q}/\mathbb{Z}$-valued quadratic form $\gamma\mapsto\frac{\gamma^{2}}{2}\in\mathbb{Q}/\mathbb{Z}$ on $D_{L}$, rendering the latter group a finite quadratic module in the sense of \cite{[Sch]} and \cite{[Str]}.

Let $(b_{+},b_{-})$ be the signature of the quadratic space $L_{\mathbb{R}}:=L\otimes\mathbb{R}$. The group $Mp_{2}(\mathbb{Z})$ admits a \emph{Weil representation} $\rho_{L}$ on $\mathbb{C}[D_{L}]$, in which the standard basis vector $\mathfrak{e}_{\gamma}$ corresponding to $\gamma \in D_{L}$ is an eigenvector of $T$ with eigenvalue $\mathbf{e}\big(\frac{\gamma^{2}}{2}\big)$, and $S$ operates via a Fourier transform:
\begin{equation}
\rho_{L}(S)\mathfrak{e}_{\gamma}=\frac{\zeta_{8}^{-sgn(L)}}{\sqrt{\Delta_{L}}}
\sum_{\delta \in D_{L}}\mathbf{e}\big(-(\gamma,\delta)\big)\mathfrak{e}_{\delta}. \label{rhoLS}
\end{equation}
Here and throughout, the symbol $\mathbf{e}(z)$ stands for $e^{2\pi iz}$ for any $z\in\mathbb{C}$, and $\zeta_{8}$ is the primitive 8th root of unity $\mathbf{e}\big(\frac{1}{8}\big)$. The previous two references, as well as \cite{[Ze1]}, give formulae for the action of a general element of $Mp_{2}(\mathbb{Z})$ via $\rho_{L}$.

\smallskip

A lattice which will be of important technical use in what follows is the rank 1 lattice $L_{1}$ generated by a single element with norm 2. Its discriminant group $D_{1}:=D_{L_{1}}$ is cyclic of order 2, and $\mathbb{C}[D_{1}]$ has the canonical basis consisting of the two vectors $\mathfrak{e}_{0}$ and $\mathfrak{e}_{1}$. Under the canonical basis, the Weil representation $\rho_{1}:=\rho_{L_{1}}$ becomes
\begin{equation}
\rho_{1}(T)=\begin{pmatrix} 1 & 0 \\ 0 & i\end{pmatrix},\quad\rho_{1}(S)=\frac{1-i}{2}\begin{pmatrix} 1 & 1 \\ 1 & -1\end{pmatrix},\quad\mathrm{and}\quad\rho_{1}(R)=\frac{1-i}{2}\begin{pmatrix} 1 & i \\ i & 1\end{pmatrix}. \label{rho1TSR}
\end{equation}
In addition, for elements $\big(\alpha,\pm\sqrt{j(\alpha,\tau)}\big)\in\widetilde{\Gamma}_{0}(4)$ the formula becomes
\begin{equation}
\rho_{1}\big(\alpha,\pm\sqrt{j(\alpha,\tau)}\big)=\psi\big(\alpha,\pm\sqrt{j(\alpha,\tau)}\big)\begin{pmatrix} 1 & 0 \\ 0 & i^{bd}\end{pmatrix},\quad\mathrm{where}\quad\psi\big(\alpha,\pm\sqrt{j(\alpha,\tau)}\big)=
\pm\bigg(\frac{c}{d}\bigg)\overline{\varepsilon_{d}}. \label{rho1Gamma04}
\end{equation}
Here $\varepsilon_{d}$ is the classical symbol which equals 1 if $d\equiv1(\mathrm{mod\ }4)$ and $i$ if $d\equiv3(\mathrm{mod\ }4)$. Note that the Legendre symbol here uses the convention of \cite{[Str]} and others (and not of \cite{[Ze1]}).

To every positive definite lattice $L$ one can associate a vector-valued theta function $\Theta_{L}$. When $L=L_{1}$, it is simply
\begin{equation}
\Theta(\tau):=\sum_{\delta\in\{0,1\}}\theta_{\delta}(\tau)\mathfrak{e}_{\delta},\quad\mathrm{with}\quad\theta_{0}(\tau)=\theta(\tau) :=\sum_{n\in\mathbb{Z}}q^{n^{2}},\quad\mathrm{and}\quad\theta_{1}(\tau)=\sum_{n\in\mathbb{Z}}q^{(n+1/2)^{2}}. \label{Theta}
\end{equation}
Here and throughout, the classical notation $q$ stands for $\mathbf{e}(\tau)$. Now, it is easily verified that
\begin{equation}
\theta(\tau)=\theta_{0}(4\tau)+\theta_{1}(4\tau),\quad\theta_{0}(\tau)=\frac{\theta\big(\frac{\tau}{4}\big)+\theta\big(\frac{\tau+2}{4}\big)}{2},
\quad\mathrm{and}\quad\theta_{1}(\tau)=\frac{\theta\big(\frac{\tau}{4}\big)-\theta\big(\frac{\tau+2}{4}\big)}{2}, \label{Thetarels}
\end{equation}
and that the non-zero Fourier coefficients of $\theta_{j}(4\tau)$ are congruent to $j$ modulo 4 for $j \in \{0, 1\}$.
One may consider the relations from Equation \eqref{Thetarels}, as well as the property of indices modulo 4, as the prototypical example for the level 1 case of Theorem \ref{rho1lift} below. The character $\psi$ is related to $\theta$ through the fact that it appears in its transformation rule: \[\theta(\alpha\tau)=\psi(\alpha,\phi)\phi(\tau)\theta(\tau)\quad\mathrm{for\ all}\quad(\alpha,\phi)\in\widetilde{\Gamma}_{0}(4).\]

\smallskip

\begin{rmk}
The element $R^{2}$, whose square is in $\widetilde{\Gamma}_{0}(4)$, generates $\widetilde{\Gamma}_{0}(2)$ with $\widetilde{\Gamma}_{0}(4)$ and preserves the latter group under conjugation. As its $\rho_{1}$-image $\big(\begin{smallmatrix} 0 & 1 \\ 1 & 0\end{smallmatrix}\big)$ simply interchanges the basis elements $\mathfrak{e}_{0}$ and $\mathfrak{e}_{1}$, the action of $\widetilde{\Gamma}_{0}(4)$ on $\mathfrak{e}_{1}$ is via the character $\psi_{R^{2}}$ sending $A\in\widetilde{\Gamma}_{0}(4)$ to $\psi(R^{2}AR^{-2})$. Combining this with the form of $\rho_{1}(R^{2})$ once more, we find that the restriction of $\rho_{1}$ to $\widetilde{\Gamma}_{0}(2)$ can be seen as the representation induced to $\widetilde{\Gamma}_{0}(2)$ from the character $\psi$ of the index 2 subgroup $\widetilde{\Gamma}_{0}(4)$. \label{indrep}
\end{rmk}

The subgroup $\ker\psi\subseteq\widetilde{\Gamma}_{0}(4)$ is a lift of $\Gamma_{1}(4)$ to $Mp_{2}(\mathbb{Z})$, and the group $\widetilde{\Gamma}_{0}(4)$ decomposes as $\langle Z \rangle\times\ker\psi$. Given any subgroup $\widetilde{\Gamma}\subseteq\widetilde{\Gamma}_{0}(4)$, we denote \[\widetilde{\Gamma}^{\psi}:=\widetilde{\Gamma}\cap\ker\psi,\quad\mathrm{so\ that\ in\ particular}\quad\ker\psi=\widetilde{\Gamma}_{0}^{\psi}(4)=\widetilde{\Gamma}_{1}^{\psi}(4).\]
Some details about the interplay between $\widetilde{\Gamma}$ and $\widetilde{\Gamma}^{\psi}$ are given below.
\begin{lem}
Suppose $\widetilde{\Gamma}\subseteq\widetilde{\Gamma}_{0}(4)$ is the double cover of some subgroup $\Gamma\subseteq\Gamma_{0}(4)$. Then the index $[\widetilde{\Gamma}:\widetilde{\Gamma}^{\psi}]$ is 2 if $\widetilde{\Gamma}\subseteq\widetilde{\Gamma}_{1}(4)$ and 4 otherwise. In the former case we have $\widetilde{\Gamma}= \widetilde{\Gamma}^{\psi} \times \langle Z^{2} \rangle$. In the latter case a decomposition of the form $\widetilde{\Gamma}=\widetilde{\Gamma}^{\psi} \times H$ is possible if and only if $Z\in\widetilde{\Gamma}$ and $H=\langle Z \rangle$. \label{Gammapsiind}
\end{lem}

\begin{proof}
We know that $\ker\psi\subseteq\widetilde{\Gamma}_{1}(4)$ does not contain $Z^{2}$. Moreover, $\psi$ takes the values $\pm1$ on $\widetilde{\Gamma}_{1}(4)$, and $\widetilde{\Gamma}\cap\widetilde{\Gamma}_{1}(4)$ still contain $Z^{2}$. These combine to prove the first assertion. The second one is immediate, since $\widetilde{\Gamma}^{\psi} \subseteq Mp_2(\mathbb{Z})$ is a lift of $\Gamma\cap\Gamma_{1}(4) \subseteq SL_2(\mathbb{Z})$. The existence of the decomposition if $Z\in\widetilde{\Gamma}$ is also clear. Conversely, one easily proves that $\pm I$ are the only two matrices in $SL_{2}(\mathbb{R})$ of order 2. Since $\psi(Z^{\pm1})$ has order 4, it follows that the order 2 element of the complementary subgroup must be $Z^{2}$, and another application of that argument shows that this subgroup, if it exists, must indeed be $\langle Z \rangle$. This proves the lemma.
\end{proof}

\smallskip

Let $\widetilde{\Gamma}$ be a discrete subgroup of $Mp_{2}(\mathbb{R})$, of finite co-volume. Let $\rho:\widetilde{\Gamma} \to GL(V)$ be a representation of $\widetilde{\Gamma}$ on a complex vector space $V$, and let $\kappa\in\frac{1}{2}\mathbb{Z}$ be a weight. For every $A=\big(\alpha,\phi(\tau)\big)\in\widetilde{\Gamma}$ we define the \emph{Petersson slash operator} $\big|_{\kappa,\rho}A$ on functions $f:\mathcal{H} \to V$ by \[\big(f\big|_{\kappa,\rho}{A}\big)(\tau)=\phi(\tau)^{-2\kappa}\rho(A)^{-1}f(\alpha\tau).\] When the weight $\kappa$ is fixed and $\rho$ is trivial, we will omit them from the notation and write $f\big|A$ instead.

\begin{defn}
In the notations above, we say that a real-analytic function $f:\mathcal{H} \to V$ is an \emph{automorphic form of weight $\kappa$ and representation $\rho$ with respect to $\widetilde{\Gamma} \subseteq Mp_{2}(\mathbb{R})$} if the equality $(f\big|_{\kappa,\rho}A)=f$ holds, as functions on $\mathcal{H}$, for all $A\in\widetilde{\Gamma}$. We denote the space of such functions by $\mathcal{A}_{\kappa}(\widetilde{\Gamma},\rho)$. Assuming further that $f$ is holomorphic on $\mathcal{H}$, we call it a \emph{weakly holomorphic modular form} if it is meromorphic at the cusps of $\widetilde{\Gamma}$ (if it has any), and a \emph{modular form} if it is holomorphic at all the cusps. As usual, we use $\mathcal{M}_{\kappa}^{!}(\widetilde{\Gamma},\rho)$ and $\mathcal{M}_{\kappa}(\widetilde{\Gamma},\rho)$ to denote the subspaces of such forms. A modular form vanishing at all the cusps is called a \emph{cusp form} (this is always the case if $\widetilde{\Gamma}$ has no cusps). In case $\rho$ is missing, it should be understood to be 1-dimensional and trivial. \label{autform}
\end{defn}

We remark that the growth conditions at the cusps in Definition \ref{autform} are not necessary for the lifting from scalar-valued to vector-valued modular forms in the following two sections. The vector-valued function $\Theta$ from Equation \eqref{Theta}, for example, belongs to $\mathcal{M}_{1/2}\big(Mp_{2}(\mathbb{Z}),\rho_{1}\big)$. For any representation $\rho:\widetilde{\Gamma} \to GL(V)$, the space $V$ decomposes according to the action of $\widetilde{\Gamma} \cap \langle Z \rangle$ via $\rho$. Investigating the action of $Z$ proves the following
\begin{prop}
The images of functions in $\mathcal{A}_{\kappa}(\widetilde{\Gamma},\rho)$
lie in the subspace of $V$ on which $\widetilde{\Gamma}\cap\langle Z \rangle$ operates via $\psi^{2\kappa}$. \label{imrhoZpsi}
\end{prop}

\begin{proof}
We may decompose $V$ to the eigenspaces of the action of the generator $Z^{r}$ of $\widetilde{\Gamma}\cap\langle Z \rangle$, and compose an element of $\mathcal{A}_{\kappa}(\widetilde{\Gamma},\rho)$ with the projection onto one such eigenspace, on which $Z^{r}$ acts as multiplication by some scalar $\varepsilon$ (of order dividing $\frac{4}{(4,r)}$). Then $f$ must be invariant under the slash operator $\big|_{k,\rho}Z^{r}$, which is easily verified to be multiplication by $\frac{i^{2r\kappa}}{\varepsilon}=\frac{\psi(Z^{r})}{\varepsilon}$. For this part of $f$ to be non-zero, this scalar must equal 1. This proves the proposition.
\end{proof}

It follows from Proposition \ref{imrhoZpsi} that if $Z^{2}$ lies in $\widetilde{\Gamma}$ and operates on $V$ as a scalar $\varepsilon\in\{\pm1\}$, then $\mathcal{A}_{\kappa}(\widetilde{\Gamma},\rho)\neq0$ only if $\varepsilon=(-1)^{2\kappa}$. In particular for $\mathcal{A}_{\kappa}(\widetilde{\Gamma},\rho)$ not to vanish for integral $\kappa$ it is necessary for $\rho$ to factor through the image $\Gamma \subseteq SL_{2}(\mathbb{R})$. In this case $\mathcal{A}_{\kappa}(\widetilde{\Gamma},\rho)$ is the same as the more classical space $\mathcal{A}_{\kappa}(\Gamma,\rho)$ (not involving metaplectic groups). On the other hand, assume that $\widetilde{\Gamma} \subseteq Mp_{2}(\mathbb{R})$ and $\kappa\in\frac{1}{2}\mathbb{Z}$ are given such that $\mathcal{A}_{\kappa}(\widetilde{\Gamma})\neq0$ (with trivial $\rho$). Proposition \ref{imrhoZpsi} implies that in this case we have the equality
\begin{equation}
\mathcal{A}_{\kappa}(\widetilde{\Gamma})=\mathcal{A}_{\kappa}\big(\langle\widetilde{\Gamma},Z\rangle,\psi^{2\kappa}\big)
=\mathcal{A}_{\kappa}\big(\langle\widetilde{\Gamma},Z^{2}\rangle,\psi^{2\kappa}\big). \label{Zpsi}
\end{equation}
Here $\langle\widetilde{\Gamma},Z^{r}\rangle$ stands for the subgroup of $Mp_{2}(\mathbb{R})$ which is generated by $\widetilde{\Gamma}$ and $Z^{r}$. We shall be applying Equation \eqref{Zpsi} only to subgroups of $\widetilde{\Gamma}_{0}^{\psi}(4)$ in this paper.

\smallskip

It is well-known that for integral $\kappa$, the space
$\mathcal{A}_{\kappa}\big(\Gamma_{1}(N)\big)$ decomposes with respect to the
action of $\Gamma_{0}(N)$ according to those characters $\chi$ of $\Gamma_{0}(N)/\Gamma_{1}(N)\cong(\mathbb{Z}/N\mathbb{Z})^{\times}$ satisfying $\chi(-1)=(-1)^{\kappa}$. The same assertion holds for $\mathcal{A}_{\kappa}\big(\Gamma_{1}(N),\rho\big)$ wherever $\rho$ is a representation of $\Gamma_{0}(N)$. The analogous assertion for $\kappa\in\frac{1}{2}\mathbb{Z}$ is as follows.
\begin{prop}
Let $\kappa\in\frac{1}{2}\mathbb{Z}$ be general. We then have
\[\mathcal{A}_{\kappa}\big(\widetilde{\Gamma}_{1}^{\psi}(4N)\big)=\bigoplus_{
\omega}\mathcal{A}_{\kappa}(\widetilde{\Gamma}_{0}(4N),\psi^{2\kappa}\omega),\]
where the direct sum is taken over the \textbf{even} Dirichlet characters $\omega$ modulo $4N$. \label{Z/2char}
\end{prop}

\begin{proof}
By Equation \eqref{Zpsi}, the left hand side is just $\mathcal{A}_{\kappa}\big(\langle \widetilde{\Gamma}_{1}(4N), Z \rangle,\psi_{}^{2\kappa}\big)$. Under the action of $\widetilde{\Gamma}_{0}(4N)$, this space decomposes according to characters of \[\widetilde{\Gamma}_{0}(4N)/\langle\widetilde{\Gamma}_{1}(4N),Z\rangle\cong\Gamma_{0}(4N)/\langle\pm I,\Gamma_{1}(4N)\rangle\cong(\mathbb{Z}/4N\mathbb{Z})^{\times}/\{\pm1\},\] which are exactly the even Dirichlet characters modulo $4N$. This proves the proposition.
\end{proof}

\smallskip

For a character $\chi$ modulo $N$, let $\omega_{\chi}$ be the character modulo $4N$ defined by
\begin{equation}
\omega_{\chi}(d):=\bigg(\frac{4\chi(-1)}{d}\bigg)\chi(d). \label{omegachi}
\end{equation}
They are important for analyzing automorphic forms with respect to the group $\widetilde{\Gamma}_{N}$ from Equation \eqref{GammaN}, as one sees in the following lemma.
\begin{lem}
The projection of $\widetilde{\Gamma}_{N}^{\psi}$ to $SL_{2}(\mathbb{Z})$ is $\Gamma_{N}$ if $4|N$ and $\Gamma_{0}(4N) \cap \Gamma_{1}(4N_{2})$ if $4 \nmid N$. Here $N_{2}$ be the odd part of $N$. Furthermore, the characters $\big\{\omega_{\chi}\mid\chi\mathrm{\ is\ a\ character\ modulo\ }N\big\}$ are exactly the even characters of $\widetilde{\Gamma}_{0}(4N)/\widetilde{\Gamma}_{1}(4N)$ that factor through $\widetilde{\Gamma}_{N}^{\psi}$. The choice of $\chi$ such that $\omega=\omega_{\chi}$ is unique if $4 \nmid N$, but if $4|N$ there are two choices with different parities. \label{charGammaN}
\end{lem}

\begin{proof}
If $4|N$ then $\Gamma_{N}\subseteq\Gamma_{1}(N)\subset\Gamma_{1}(4)$ and Lemma \ref{Gammapsiind} shows that $\widetilde{\Gamma}_{N}^{\psi}$ is an isomorphic lift of $\Gamma_{N}$ in $Mp_{2}(\mathbb{Z})$. When $4 \nmid N$, Lemma \ref{Gammapsiind} implies that $\widetilde{\Gamma}_{N}^{\psi}$ lies over $\Gamma_{N}\cap\Gamma_{1}(4)$, which has index 2 in $\Gamma_{N}$. Now, an element of ${\Gamma}_{0}(4N)\cap{\Gamma}_{1}(N_{2})$ immediately lies in ${\Gamma}_{1}(2N_{2})$, and it lies in ${\Gamma}_{1}(4)$ if and only if it lies in ${\Gamma}_{1}(4N_{2})$. This proves the first assertion.

Now, any even character $\omega$ modulo $4N$ factors through $\widetilde{\Gamma}_{N}^{\psi}$ if and only if it factors through the projection of $\widetilde{\Gamma}_{N}^{\psi}$ to $SL_{2}(\mathbb{Z})$. In case $4|N$ this means that $\omega$ factors through $\Gamma_{N}$, i.e. it is a character modulo $N$. We can then write $\omega$ as either $\omega_{\omega}$ (this means $\omega_{\chi}$ with $\chi$ being the even character $\omega$ modulo $N$), or as $\omega_{\chi}$ for the (odd) character $\chi$ modulo $N$ which is defined by $\chi(d)=\omega(d)\big(\frac{-1}{d}\big)$. When $4 \nmid N$, it is clear that $\omega_{\chi}$ is an even character that factors through $\Gamma_{0}(4N) \cap \Gamma_{1}(4N_{2})$. On the other hand, any even character $\omega$ modulo $4N$ can be written as $\omega_{2} \cdot \chi$ with $\chi$ modulo $N_{2}$ and $\omega_{2}$ modulo 8. Then $\omega$ factors through $\Gamma_{0}(4N) \cap \Gamma_{1}(4N_{2})$ if and only if $\omega_{2}$ has conductor $4$ or 1, in which cases the evenness of $\omega$ implies that $\omega=\omega_{\chi}$ for this unique $\chi$. This completes the proof of the lemma.
\end{proof}

We can now decompose the space of automorphic forms with respect to $\widetilde{\Gamma}_{N}^{\psi}$.

\begin{prop}
Given any $N\in\mathbb{N}$, the equality
\[\mathcal{A}_{\kappa}(\widetilde{\Gamma}_{N}^{\psi})=
\begin{cases}
\mathcal{A}_{\kappa}(\widetilde{\Gamma}_{N},\psi^{2\kappa})\oplus\mathcal{A}_{\kappa}(\widetilde{\Gamma}_{N},\psi^{2\kappa-2}),  & 4 \nmid N, \\
\mathcal{A}_{\kappa}(\widetilde{\Gamma}_{N},\psi^{2\kappa})=\mathcal{A}_{\kappa}(\widetilde{\Gamma}_{N},\psi^{2\kappa-2}), & 4|N \end{cases}\]
holds, where for $\xi\in\{\pm1\}$ we have
\[\mathcal{A}_{\kappa}(\widetilde{\Gamma}_{N},\psi^{2\kappa+\xi-1})=\bigoplus_{\chi(-1)=\xi}
\mathcal{A}_{\kappa}(\widetilde{\Gamma}_{0}(4N),\psi^{2\kappa}\omega_{\chi}).\]
Here $\chi$ in the direct sum is understood to be a character modulo $N$, with the required parity. \label{char4notdivN}
\end{prop}

\begin{proof}
If $4|N$ then $\widetilde{\Gamma}_{N}=\widetilde{\Gamma}_{N}^{\psi}\times\langle Z^{2} \rangle\subseteq\widetilde{\Gamma}_{1}(4)$ by Lemmas \ref{charGammaN} and \ref{Gammapsiind}. Hence $\mathcal{A}_\kappa(\widetilde{\Gamma}_{N}^{\psi})=\mathcal{A}_\kappa(\widetilde{\Gamma}_{N} ,\psi^{2\kappa})$ by Equation \eqref{Zpsi}, and we may replace $\psi^{2\kappa}$ by $\psi^{2\kappa-2}$ since $\psi^{2}$ is trivial on the subgroup $\widetilde{\Gamma}_{N}$ of $\widetilde{\Gamma}_{1}(4)$. On the other hand, when $4 \nmid N$ the group $\widetilde{\Gamma}_{N}$ is not contained in $\widetilde{\Gamma}_{1}(4)$ since the matrix $\alpha_{N}:=\big(\begin{smallmatrix} 1+2N_{2} & * \\ 4N & 1+2N_{2}\end{smallmatrix}\big)$ lies in $\Gamma_{N}\setminus\Gamma_{1}(4)$. Hence we have $[\widetilde{\Gamma}_{N}:\widetilde{\Gamma}_{N}^{\psi}]=4$ by Lemma \ref{Gammapsiind}, and $\psi^{2}$ is non-trivial on $\widetilde{\Gamma}_{N}$ in this case (indeed, it takes the value $-1$ on the metaplectic pre-images of $\alpha_{N}$). The spaces with the different $\xi$s are therefore disjoint, and are separated by the action of the generator $\alpha_{N}$ of $\widetilde{\Gamma}_{N}/\langle\widetilde{\Gamma}_{N}^{\psi},Z^{2}\rangle$. This proves the first equation.

We now apply Proposition \ref{Z/2char} and the second assertion of Lemma \ref{charGammaN} in order to decompose $\mathcal{A}_{\kappa}(\widetilde{\Gamma}_{N}^{\psi})$ according to the even characters modulo $4N$ which take the form $\omega_{\chi}$ for some character $\chi$ modulo $N$. If $4|N$ then the second equality follows immediately from the third assertion of Lemma \ref{charGammaN}. On the other hand, if $4 \nmid N$ then this assertion shows that the direct sum over $\omega_{\chi}$ is the direct sum over $\chi$. Now, the value of $\xi$ such that $\mathcal{A}_{\kappa}(\widetilde{\Gamma}_{0}(4N),\psi^{2\kappa}\omega_{\chi})$ is contained in $\mathcal{A}_{\kappa}(\widetilde{\Gamma}_{N},\psi^{2\kappa+\xi-1})$ is determined by the action of our matrix $\alpha_{N}$. But any character $\chi$ modulo $N$ is trivial on our matrix $\alpha_{N}$, while the other factor appearing in the definition of $\omega_{\chi}$ in Equation \eqref{omegachi} takes this matrix to $\chi(-1)$. Therefore the character $\psi^{2\kappa}\omega_{\chi}$ coincides with $\psi^{2\kappa+\chi(-1)-1}$ on $\alpha_{N}$, showing that indeed for each character $\chi$ we have to take $\xi=\chi(-1)$. This completes the proof of the proposition.
\end{proof}

The fact that in the case $4|N$ the same space can be described using both values of $\xi$, a fact which may seem a bit cumbersome at this point, will turn out very useful below.

\begin{rmk}
If $N$ is 1 or 2 then $Z\in\widetilde{\Gamma}_{N}$, making $\mathcal{A}_\kappa(\widetilde{\Gamma}_{N},\psi^{2\kappa-2})$ the trivial subspace. This agrees with the fact that there are no odd characters with these moduli. Lemma \ref{Gammapsiind} shows that these are also the only cases where $\widetilde{\Gamma}_{N}$ contains $\widetilde{\Gamma}_{N}^{\psi}$ as a direct factor of index 4. \label{SVtriv}
\end{rmk}

\section{Liftings to a 2-Dimensional Representation \label{2dLift}}

Let $N$ and $k$ be positive integers and $\xi\in\{\pm1\}$. Inside $\mathcal{A}_{k+1/2}(\widetilde{\Gamma}_{N},\psi^{2k+\xi})$ we consider the ``plus-space'' \[\mathcal{A}^{+}_{k+1/2}(\widetilde{\Gamma}_{N},\psi^{2k+\xi}):=\bigg\{\sum_{n\in\mathbb{Z}} c_{n}(y)q^{n}\in\mathcal{A}_{k+1/2}(\widetilde{\Gamma}_{N},\psi^{2k+\xi})\bigg|c_{n}(y)=0\mathrm{\ unless\ }(-1)^{k}\xi n\equiv 0,1(\mathrm{mod\ }4)\bigg\}.\] In addition, given a character $\chi$ modulo $N$ satisfying $\chi(-1)=\xi$ we define
\[\mathcal{A}_{k+1/2}^{+}\big(\widetilde{\Gamma}_{0}(4N),\psi^{2k+1}\omega_{\chi}\big):=\mathcal{A}_{k+1/2}
\big(\widetilde{\Gamma}_{0}(4N),\psi^{2k+1}\omega_{\chi}\big)\cap\mathcal{A}_{k+1/2}^{+}(\widetilde{\Gamma}_{N},\psi^{2k+\xi}).\]
Similarly, adding the superscript ``+'' to any of the spaces $\mathcal{M}_{k+1/2}^{!}(\widetilde{\Gamma}_{N},\psi^{2k+\xi})$, $\mathcal{M}_{k+1/2}(\widetilde{\Gamma}_{N},\psi^{2k+\xi})$, or $\mathcal{M}_{k+1/2}(\widetilde{\Gamma}_{0}(4N),\psi^{2k+1}\omega_{\chi})$ denotes its intersection with $\mathcal{A}^{+}_{k+1/2}(\widetilde{\Gamma}_{N},\psi^{2k+\xi})$.

For each $f(\tau)=\sum_{n\in\mathbb{Z}}c_{n}(y)q^{n}\in\mathcal{A}_{k+1/2}(\widetilde{\Gamma}_{N},\psi^{2k+\xi})$ and any $j$ modulo 4 we define
\begin{equation}
f_{j}(\tau):=\sum_{n \equiv j(\mathrm{mod\ }4)}c_{n}\big(\tfrac{y}{4}\big)q^{n/4}. \label{fjdef}
\end{equation}
We now prove
\begin{lem}
The modular form $f$ and the functions $f_{j}$ from Equation \eqref{fjdef} are related through
\[f(\tau)=\sum_{j\in\mathbb{Z}/4\mathbb{Z}}f_{j}(4\tau),\quad
\frac{f\big(\frac{\tau}{4}\big)+f\big(\frac{\tau+2}{4}\big)}{2}=f_{0}(\tau)+f_{2}(\tau),\quad\mathrm{and}\quad
\frac{f\big(\frac{\tau}{4}\big)-f\big(\frac{\tau+2}{4}\big)}{2}=f_{1}(\tau)+f_{-1}(\tau).\] If $f(\tau)$ is in the plus-space $\mathcal{A}^{+}_{k+1/2}(\widetilde{\Gamma}_{N},\psi^{2k+\xi})$ then the right hand sides of the latter equation become just $f_{0}(4\tau)+f_{\varepsilon}(4\tau)$, $f_{0}(\tau)$, and $f_{\varepsilon}(\tau)$ respectively, where $\varepsilon=(-1)^{k}\xi$. \label{cnmod4}
\end{lem}

\begin{proof}
The proof is a straightforward calculation.
\end{proof}

\begin{rmk}
With our value of $\varepsilon$ we can also write $\psi^{2k+\xi}$ as $\psi^{\varepsilon}$, since $\psi$ has order 4. \label{psi2k+xieps}
\end{rmk}

In case $4 \nmid N$ the space $\mathcal{A}^{+}_{k+1/2}(\widetilde{\Gamma}_{N},\psi^{2k+\xi})$ (or more precisely its subspaces $\mathcal{A}^{+}_{k+1/2}(\widetilde{\Gamma}_{0}(4N),\psi^{2k+1}\omega_{\chi})$)
is known, at least for $\xi=(-1)^{k}$ (i.e., $\varepsilon=1$), as Kohnen's plus-space. It has been studied in \cite{[K2]} and \cite{[UY]}, from which we know that there are appropriate projection operators from $\mathcal{A}_{k+1/2}$ to $\mathcal{A}^{+}_{k+1/2}$. These projection operators are essentially based on the matrix $T^{1/4}=\big(\begin{smallmatrix} 1 & 1/4 \\ 0 & 1\end{smallmatrix}\big)$. When $4|N$ the situation is, in fact, much simpler, as one sees in the following
\begin{prop}
Suppose $4|N$. Choose any $\xi\in\{\pm1\}$ and $f\in\mathcal{A}_{k+1/2}(\widetilde{\Gamma}_{N}^{\psi})$, and set $\varepsilon=(-1)^{k}\xi$. Then the function taking $\tau\in\mathcal{H}$ to $f_{0}(4\tau)+f_{\varepsilon}(4\tau)$ lies in $\mathcal{A}^{+}_{k+1/2}(\widetilde{\Gamma}_{N},\psi^{2k+\xi})$, and is the image of $f$ under an appropriate projection operator $P_{\varepsilon}$.
A similar projection $P_{2}$ given by $P_{2}f(\tau):=f_{0}(4\tau)+f_{2}(4\tau)$
also preserves the space $\mathcal{A}_{k+1/2}(\widetilde{\Gamma}_{N}^{\psi})$. \label{proj+4divN}
\end{prop}

\begin{proof}
One verifies by simple conjugation that if $4 \mid N$ then $T^{1/4}$ and its powers normalize $\Gamma_{N}$ as well as its lift $\widetilde{\Gamma}_{N}^{\psi}$ to $Mp_{2}(\mathbb{Z})$. It follows that the function $f_{j}(4\tau)$, which can be written as  $\sum_{r\in\mathbb{Z}/4\mathbb{Z}}\frac{i^{-rj}}{4}f\big|T^{r/4}$, lies in $\mathcal{A}_{k+1/2}(\widetilde{\Gamma}_{N}^{\psi})$ for each $j\in\mathbb{Z}/4\mathbb{Z}$ all 4 values of $j$. The combinations $P_{\varepsilon}f$ and $P_{2}$ therefore also lie in this space, and the plus-space condition for $P_{\varepsilon}f$ is clear. This proves the proposition.
\end{proof}
The explicit formulae for the projection operators $P_{\varepsilon}$ and $P_{2}$ from Proposition \ref{proj+4divN} in terms of the slash operator are using the combinations
$\frac{1}{2}I+\frac{\overline{\zeta_{8}}}{2\sqrt{2}}T^{\varepsilon/4}+\frac{\zeta_{8}}{2\sqrt{2}}T^{-\varepsilon/4}$
and $\frac{1}{2}I+\frac{1}{2}T^{1/2}$ respectively. A small follow-up of the argument proving Proposition \ref{proj+4divN} also shows that the plus-space is in general non-trivial when $4|N$.

\smallskip

The main result of this section can now be stated and proved. We let $\rho_{\varepsilon}$ denote $\rho_{1}$ if $\varepsilon=+1$, while for $\varepsilon=-1$ we consider the dual (or complex conjugate) representation $\rho_{1}^{*}=\overline{\rho}_{1}$ associated with the discriminant group $D_{1}(-1)$ of the lattice $L_{1}(-1)$ obtained from $D_{1}$ and $L_{1}$ by multiplying the bilinear and quadratic forms by $-1$. The canonical basis elements of $\rho_{1}^{*}$ will be denoted $\mathfrak{e}_{0}$ and $\mathfrak{e}_{-1}$.
\begin{thm}
Let $N$, $k$, $\xi$, $\varepsilon$, and $f\in\mathcal{A}_{k+1/2}^{+}(\widetilde{\Gamma}_{N},\psi^{2k+\xi})$ be given as above. Then the vector-valued function
\begin{equation}
\mathcal{L}_{\varepsilon}f(\tau):=f_{0}(\tau)\mathfrak{e}_{0}+f_{\varepsilon}(\tau)\mathfrak{e}_{\varepsilon},
\label{Lepsdef}
\end{equation}
in which the coefficients are defined in Equation \eqref{fjdef}, lies in $\mathcal{A}_{k+1/2}\big(\widetilde{\Gamma}_{1}(N),\rho_{\varepsilon}\big)$, and the map $\mathcal{L}_{\varepsilon}$ is an isomorphism between these two spaces. Moreover, if $\chi$ is a Dirichlet character modulo $N$ satisfying $\chi(-1)=\xi$ then image of the subspace $\mathcal{A}^{+}_{k+1/2}(\widetilde{\Gamma}_{0}(4N),\psi^{2k+1}\omega_{\chi})$ of $\mathcal{A}^{+}_{k+1/2}(\widetilde{\Gamma}_{0}(4N),\psi^{2k+\xi})$ under the map $\mathcal{L}_{\varepsilon}$ is exactly $\mathcal{A}_{k+1/2}(\widetilde{\Gamma}_{0}(N),\chi\otimes\rho_{\varepsilon})$.
\label{rho1lift}
\end{thm}

\begin{rmk}
The first part of Equation \eqref{Thetarels} shows that the map $\mathcal{L}_{+}$ takes the function $\theta\in\mathcal{A}_{1/2}^{+}(\widetilde{\Gamma}_{1},\psi)$ to the vector-valued theta function $\Theta$ from Equation \eqref{Theta}. The other parts of Equation \eqref{Thetarels} are the corresponding special case of Lemma \ref{cnmod4}. In fact, one may use the relations between $\theta$ and $\Theta$ appearing in Equation \eqref{Thetarels} in order to establish some of the equalities appearing in the proof below. \label{thetaspcase}
\end{rmk}

\begin{proof}
We will give only the proof for $\varepsilon=1$, as the other case is established by the same argument, with the changes indicated in Remark \ref{L-} below. The modularity of $\mathcal{L}_{+}f$ with respect to $T$ follows immediately from Equation \eqref{rho1TSR} and the Fourier expansions of $f_{0}(\tau)$ and $f_{1}(\tau)$. As $\widetilde{\Gamma}_{1}(N)$ is generated by $T$ and $\widetilde{\Gamma}_{1}(N)\cap\widetilde{\Gamma}^{0}(4)$, it suffices to investigate the action elements of the latter group.

This investigation becomes more convenient when one uses the basis of $\mathbb{C}[D_{1}]$ consisting of the vectors $\mathfrak{e}_{+}:=\frac{\mathfrak{e}_{0}+\mathfrak{e}_{1}}{\sqrt{2}}$ and $\mathfrak{e}_{-}:=\frac{\mathfrak{e}_{0}-\mathfrak{e}_{1}}{\sqrt{2}}$. We consider, for $j\in\{0,1\}$, the matrix $\alpha_{j}:=\big(\begin{smallmatrix} 1/2 & j \\ 0 & 2\end{smallmatrix}\big)$ and the metaplectic element $A_{j}:=(\alpha_{j},\sqrt{2}) \in Mp_{2}(\mathbb{R})$. Writing $f_{0}$ and $f_{1}$ as in Lemma \ref{cnmod4} and using the fact that $f\big(\frac{\tau+2j}{4}\big)$ can be written as $2^{k+1/2}f\big|A_{j}(\tau)$ for our two values of $j$ allows us to write
\[\mathcal{L}_{+}f(\tau)=f\big(\tfrac{\tau}{4}\big)\frac{\mathfrak{e}_{0}+\mathfrak{e}_{1}}{2}+f\big(\tfrac{\tau+2}{4}\big)
\frac{\mathfrak{e}_{0}-\mathfrak{e}_{1}}{2}=2^{k+1}\big[(f\big|A_{0})(\tau)\mathfrak{e}_{+}+(f\big|A_{1})(\tau)\mathfrak{e}_{-}\big].\]

We now observe that $\widetilde{\Gamma}_{1}(N)\cap\widetilde{\Gamma}^{0}(4)$ is generated by $R^{N}$ and the subgroup $\widetilde{\Gamma}_{1}(4N)\cap\widetilde{\Gamma}^{0}(4)$, and we begin by considering an element $A=(\alpha,\phi)$ in the latter group. As $\rho_{1}(A)$ is just the scalar $\psi(A)$ by Equation \eqref{rho1Gamma04}, we have to verify that applying the slash operator $\big|A$ to $\mathcal{L}_{+}f$ simply multiplying the latter function by $\psi(A)$. But conjugating $\alpha=\big(\begin{smallmatrix} a & b \\ c & d\end{smallmatrix}\big)$ by our $\alpha_{j}$ yields the matrix $\big(\begin{smallmatrix} a+2cj & (b+2jd-2ja)/4-cj^{2} \\ 4c & d-2cj\end{smallmatrix}\big)\in\Gamma_{1}(4N)$, and the metaplectic data in $A_{j}AA_{j}^{-1}$ has the same image as $\phi$. It follows that when we evaluate the three multipliers appearing in the definition of $\psi(A_{j}AA_{j}^{-1})$ in Equation \eqref{rho1Gamma04} we get the same value as for $\psi(A)$. Equation \eqref{rho1Gamma04} then tells us that
\[\mathcal{L}_{+}f\big|A=2^{k+1}\big[(f\big|A_{0}A)\mathfrak{e}_{+}+(f\big|A_{1}A)\mathfrak{e}_{-}\big]=
2^{k+1}\big[\psi(A_{0}AA_{0}^{-1})(f\big|A_{0})\mathfrak{e}_{+}+\psi(A_{1}AA_{1}^{-1})(f\big|A_{1})\mathfrak{e}_{-}\big],\] which indeed coincides with $\psi(A)\mathcal{L}_{+}f$ as desired.

It remains to consider the action of $R^{N}$. Equation \eqref{rho1TSR} implies that our basis vectors $\mathfrak{e}_{\pm}$ are eigenvectors for $\rho_{1}(R)$ with the respective eigenvalues being $\frac{(1-i)(1 \pm i)}{2}$.
This is 1 for $\mathfrak{e}_{+}$ and $-i$ for $\mathfrak{e}_{-}$.
We must therefore compare $\mathcal{L}_{+}f\big|R^{N}$ with the expression in which $f\big|A_{0}$ remains invariant but $f\big|A_{1}$ is multiplied by $(-i)^{N}$. But the argument from above shows that $A_{j}R^{N}A_{j}^{-1}\in\widetilde{\Gamma}_{N}$ and $f\big|A_{j}R^{N}=\psi(A_{j}R^{N}A_{j}^{-1})f\big|A_{j}$ for each $j$. Now, conjugation yields \[A_{0}RA_{0}^{-1}=R^{4}\in\ker\psi,\quad\mathrm{as\ well\ as}\quad A_{1}RA_{1}^{-1}=TR^{-4}Z.\] It follows that $\psi(A_{0}R^{N}A_{0}^{-1})=1$ and $\psi(A_{1}R^{N}A_{1}^{-1})=(-i)^{N}$, and continuing with the same argument from above indeed evaluates $\mathcal{L}_{+}f\big|A$ as the desired expression $2^{k+1}\big[(f\big|A_{0}A)\mathfrak{e}_{+}+(-i)^{N}(f\big|A_{1}A)\mathfrak{e}_{-}\big]$. This establishes the first part.

For the converse, note that our argument actually proves that the vector $\mathfrak{e}_{+}$ is an eigenvector for all the operators $\rho_{1}(B)$ with $B\in\widetilde{\Gamma}^{0}(4)$, the eigenvalue being $\psi_{A_{0}^{-1}}(B):=\psi(A^{-1}_{0}BA_{0})$ (where the argument of $\psi$ indeed lies in $\widetilde{\Gamma}_{0}(4)$). In fact, $\mathfrak{e}_{-}$ is also an eigenvector, belonging to the eigenvalue $\psi_{A_{0}^{-1}}(T^{-2}AT^{2})$ since $A_{1}=A_{0}T^{2}$. Let now an element $g_{0}\mathfrak{e}_{0}+g_{1}\mathfrak{e}_{1}\in\mathcal{A}_{k+1/2}(\widetilde{\Gamma}_{1}(N),\rho_{1})$ be given, and denote the coefficient $\frac{g_{0}+g_{1}}{\sqrt{2}}$ of $\mathfrak{e}_{+}$ by $g_{+}$. This coefficient lies in $\mathcal{A}_{k+1/2}(\widetilde{\Gamma}_{1}(N)\cap\widetilde{\Gamma}^{0}(4),\psi_{A_{0}^{-1}})$. The function $g(\tau):=g_{0}(4\tau)+g_{1}(4\tau)$, which can be written as $2^{-k}(g_{+}\big|A_{0}^{-1})(\tau)$, is invariant under $T$. As before, we can restrict our attention to the subgroup $\widetilde{\Gamma}_{1}(N)\cap\widetilde{\Gamma}^{0}(4)$. But we have $A_{0}\widetilde{\Gamma}_{N}A_{0}^{-1}=\widetilde{\Gamma}_{1}(N)\cap\widetilde{\Gamma}^{0}(4)$, which implies, as above, that \[g\big|A=2^{-k}g_{+}\big|A_{0}A=2^{-k}\big(g_{+}\big|A_{0}AA_{0}^{-1}\big)\big|A_{0}=
\psi_{A_{0}^{-1}}(A_{0}AA_{0}^{-1})\cdot2^{-k}g_{+}\big|A_{0}=\psi(A)g\] for every $A\in\widetilde{\Gamma}_{N}$. This shows that that $g\in\mathcal{A}^{+}_{k+1/2}(\widetilde{\Gamma}_{N},\psi)$. As this map is clearly a two-sided inverse to $\mathcal{L}_{+}$, we deduce that $\mathcal{L}_{+}$ is an isomorphism. The claim involving characters follows from the same of calculations, but with $\widetilde{\Gamma}_{1}(N)$ replaced by $\widetilde{\Gamma}_{0}(N)$, using the fact that $\chi(d)\psi(A)$ equals $\omega_{\chi}(d)\psi^{2k+1}(A)$ for every $A\in\widetilde{\Gamma}_{0}(4N)$ with lower right entry $d$ and Dirichlet character $\chi$ modulo $N$ which satisfies $\chi(-1)=(-1)^{k}$. This completes the proof of the theorem.
\end{proof}

\begin{rmk}
We shall use mainly the case $\varepsilon=+1$ in this paper, though some auxiliary results will need the other case as well. The case $\varepsilon=-1$ may also turn out useful when one considers lifts from modular forms to Siegel modular forms. The proof is obtained as for $\varepsilon=+1$, after replacing $\rho_{1}$ by its dual representation, conjugating all the explicit coefficients (in particular, the $\psi$-values), and considering the characters $\chi$ satisfying $\chi(-1)=-(-1)^{k}$ (where $\omega_{\chi}(d)\psi^{2k+1}(A)$ equals $\chi(d)\overline{\psi}(A)$) in the end. A proof using theta functions for $\varepsilon=-1$, as indicated in Remark \ref{thetaspcase}, may be given as well, using the anti-holomorphic complex conjugate of the theta functions from Equations \eqref{Theta} and \eqref{Thetarels}. \label{L-}
\end{rmk}

\begin{rmk}
Theorem 5.1 of \cite{[EZ]} with $m=1$ relates the space $\mathcal{M}_{k+1/2}\big(Mp_{2}(\mathbb{Z}),\rho_{-1}\big)$ to the space of Jacobi forms of weight $k$ and index 1 on the full modular group (note that our $k$ is one less than $k$ in that reference). Similarly, the space $\mathcal{M}_{k+1/2}\big(Mp_{2}(\mathbb{Z}),\rho_{1}\big)$ is related to skew-holomorphic Jacobi forms. The functions considered in our Theorem \ref{rho1lift} can be related to more general Jacobi forms of weight $k+1$ and index 1. They are holomorphic in $z$ but in general not in $\tau$. Their transformation rules are again holomorphic for $\rho_{-1}$ and skew-holomorphic for $\rho_{1}$, but now only with respect to the congruence group $\Gamma_{1}(N)$. \label{Jacobi}
\end{rmk}

\begin{rmk}
Recall that if $4 \nmid N$ then the spaces $\mathcal{A}_{k+1/2}^{+}(\widetilde{\Gamma}_{N},\psi^{\varepsilon})$ with $\varepsilon\in\{\pm1\}$ are disjoint (Proposition \ref{char4notdivN}), while if $4|N$ then the $\mathcal{A}^{+}$ condition is different for the two values of $\varepsilon$. Hence the only functions on which both operators $\mathcal{L}_{\pm}$ are defined are elements of $\mathcal{A}_{k+1/2}^{+}(\widetilde{\Gamma}_{N}^{\psi})$ with $4|N$ which are $T^{1/4}$-invariant. But in this case we have $f\big|A_{0}=f\big|A_{1}$, so that Theorem \ref{rho1lift} shows that both lifts send such a function $f$ to $2^{k+1/2}f\big|A_{0}\mathfrak{e}_{0}$, with only one component present. Now, the modularity with respect to $\widetilde{\Gamma}_{1}(N)\subseteq\widetilde{\Gamma}_{1}(4)$ and representation $\rho_{\varepsilon}$ provided by Theorem \ref{rho1lift} reduces to scalar-valued modularity with character $\psi^{\varepsilon}$, and these two powers of $\psi$ coincide on $\widetilde{\Gamma}_{1}(4)$. Hence in the only case where both lifts $\mathcal{L}_{\pm}$ are defined, the essentially coincide. \label{L+L-same}
\end{rmk}

\begin{rmk}
When $N$ is 1 or 2 the group $\widetilde{\Gamma}_{1}(N)$ contains $Z$. Equation \eqref{rho1Gamma04} shows that $\rho_{1}(Z)$ is the scalar $\psi(Z)=-i$, hence $\rho_{1}^{*}(Z)$ is the scalar $i$. From Proposition \ref{imrhoZpsi} we thus know that the space $\mathcal{A}_{k+1/2}(Mp_{2}(\mathbb{Z}), \rho_{1})$ for odd $k$, as well as the space $\mathcal{A}_{k+1/2}(Mp_{2}(\mathbb{Z}),\rho_{1}^{*})$ for even $k$, is trivial. This is in correspondence, via Theorem \ref{rho1lift}, with Remark \ref{SVtriv}. \label{VVtriv}
\end{rmk}

\section{Representations of Dimension $2N^{2}$ \label{2N^2Lift}}

For $N\in\mathbb{N}$, we let $L_{B}(N)$ be the lattice which is obtained from a hyperbolic plane through rescaling by $N$, i.e., if $E$ and $F$ are two isotropic vectors pairing to 1 generating a hyperbolic plane (so that $Q(E)=Q(F)=0$ and $(E,F)=1$) then $L_{B}(N):=\mathbb{Z} \cdot E+N\mathbb{Z} \cdot F$. The dual lattice $L_{B}^{*}(N)$ (resp.\ the discriminant group $D_{B}(N)$) is then spanned over $\mathbb{Z}$ (resp.\ $\mathbb{Z}/N\mathbb{Z}$) by $\frac{1}{N}E$ and $F$. We denote the associated Weil representation, which factors through $SL_{2}(\mathbb{Z})$, by $\rho_{B}(N)$. Given two elements $c$ and $r$ in $\mathbb{Z}/N\mathbb{Z}$, we denote the canonical basis vector $\mathfrak{e}_{cF+\frac{r}{N}E+L_{B}(N)}$ of $\mathbb{C}[D_{B}(N)]$ by $\mathfrak{e}_{c,r}$.

Now, Lemma 2.6 of \cite{[B]} constructs a map from scalar-valued modular forms
of integral weight with respect to $\Gamma_{1}(N)$ to vector-valued modular forms with representation $\rho_{B}(N)$. We generalize this map and its properties in the following proposition.
\begin{prop}
Let $\kappa\in\frac{1}{2}\mathbb{Z}$, and let $\rho:Mp_{2}(\mathbb{Z}) \to
GL(V)$ be a representation such that $\rho(Z^{2})$ operates as the scalar $(-1)^{2\kappa}$. Given any element
$\varphi\in\mathcal{A}_{\kappa}\big(\widetilde{\Gamma}_{1}(N),\rho\big)$, the
function $\mathcal{L}_{N}\varphi:\mathcal{H} \to V\otimes\mathbb{C}[D_{B}(N)]$ defined by
\begin{equation}
\mathcal{L}_{N}\varphi:=\sum_{c,r\in\mathbb{Z}/N\mathbb{Z}}(\mathcal{L}_{N}\varphi)_{c,r}\otimes\mathfrak{e}_{c,r},\quad
(\mathcal{L}_{N}\varphi)_{c,r}:=\sum_{\substack{d\in\mathbb{Z}/N\mathbb{Z} \\
(c,d,N)=1}}\mathbf{e}\big(-\tfrac{dr}{N}\big)\varphi\big|_{\kappa,\rho}A_{c,d} \label{LNdef}
\end{equation}
lies in $\mathcal{A}_{\kappa}\big(Mp_{2}(\mathbb{Z}),\rho\otimes\rho_{B}(N)\big)$, where $A_{c,d}$ is any element of $Mp_{2}(\mathbb{Z})$ whose projection to $SL_{2}(\mathbb{Z})$ has lower row $(c\ \ d)$. \label{LNphi}
\end{prop}

\begin{rmk}
The condition $\rho(Z^{2})=(-1)^{2\kappa}$ and the invariance of $\varphi$ under $\big|_{\kappa,\rho}A$ for $A\in\widetilde{\Gamma}_{1}(N)$ ensure that
$\varphi\big|_{\kappa,\rho}A_{c,d}$ only depends on $c$ and $d$ modulo $N$. In particular we may consider the action of matrices, rather than metaplectic elements. \label{LNinv}
\end{rmk}

\begin{proof}
It suffices to verify the assertion for $T$ and $S$, as they generate $Mp_{2}(\mathbb{Z})$. For $T$ note that Remark \ref{LNinv} allows us to write
$\varphi\big|_{\kappa,\rho}A_{c,d}T$ as $\varphi\big|_{\kappa,\rho}A_{c,c+d}$. After changing the summation index from $d$ to $c+d$ in the definition of $(\mathcal{L}_{N}\varphi)_{c,r}$ in Equation \eqref{LNdef} and factoring out $\mathbf{e}\big(\frac{cr}{N}\big)$, we find that the operator $\big|_{\kappa,\rho}T$ multiplies $(\mathcal{L}_{N}\varphi)_{c,r}$ by $\mathbf{e}\big(\frac{cr}{N}\big)$. Since $\mathbf{e}(\frac{cr}{N}) = \mathbf{e}(\frac{\gamma^{2}}{2})$ for the element $\gamma=cF+\frac{r}{N}E+L_{B}(N)$ of $D_{B}(N)$, it follows that the vector-valued function $\mathcal{L}_{N}\varphi$ from Equation \eqref{LNdef} is invariant under the operator $\big|_{\kappa,\rho\otimes\rho_B(N)}T$.

When considering $S$, we can use a simple exponential sum identity to write
\[\varphi\big|_{\kappa,\rho}A_{c,d}S=\varphi\big|_{\kappa,\rho}A_{d,-c}=\frac{1}{N}
\sum_{h,g\in\mathbb{Z}/N\mathbb{Z}}\mathbf{e}\big(-\tfrac{g(c+h)}{N}\big)\varphi\big|_{\kappa,\rho}A_{d,h}.\] Note that the conditions $(c,d,N)=1$ and $(d,h,N)=1$ are equivalent since the sum over $g$ vanishes unless $c=-h$. Hence $A_{d,h}$ indeed makes sense. Summing over $d\in\mathbb{Z}/N\mathbb{Z}$ with this condition yields
\[(\mathcal{L}_{N}\varphi)_{c,r}\big|_{\kappa,\rho}S=\frac{1}{N}
\sum_{d,g\in\mathbb{Z}/N\mathbb{Z}}\mathbf{e}\bigg(-\frac{cg+dr}{N}\bigg)(\mathcal{L}_{N}\varphi)_{d,g},\] where the sum over $h$ enters the definition of $(\mathcal{L}_{N}\varphi)_{d,g}$ in Equation \eqref{LNdef}. The multiplier $\frac{1}{N}$ equals $\frac{\zeta_{8}^{-sgn(L_{B}(N))}}{\sqrt{|D_{B}(N)|}}$. In addition, $\mathbf{e}\big(-\frac{cg+dr}{N}\big)$ is $\mathbf{e}\big(-(\gamma,\delta)\big)$ with $\gamma=cF+\frac{r}{N}E+L_{B}(N)$ and $\delta=dF+\frac{g}{N}E+L_{B}(N)$ in $D_{B}(N)$. Comparing with the definition of $\rho_{B}(N)(S)$ in \eqref{rhoLS}, we see that $\mathcal{L}_{N}\varphi$ is invariant under $\big|_{\kappa,\rho\otimes\rho_B(N)}S$ as well. This completes the proof of the proposition.
\end{proof}

In case $\rho(T)$ is of finite order, $\mathcal{L}_{N}\varphi$ will have a Fourier expansion. It turns out useful to express it in terms of the Fourier expansions of $\varphi$ and its images under slash operators. For this purpose we note that Equation \eqref{LNdef} and Proposition \ref{LNphi} have the following corollary.
\begin{cor}
Assume that $\rho$ is unitary, and that $u \in V$ is an eigenvector of $\rho(T)$, with eigenvalue $\mathbf{e}(\alpha)$ for some $\alpha\in\mathbb{Q}/\mathbb{Z}$. Given any $c$ and $d$ in $\mathbb{Z}/N\mathbb{Z}$ with $(c,d,N)=1$, the coefficient $\varphi_{c,d,u}$ of $u$ in $\varphi\big|_{\kappa,\rho}A_{c,d}$, in any orthogonal basis for $V$ containing $u$, admits a Fourier expansion in $q_{w}:=\mathbf{e}\big(\frac{\tau}{w}\big)$, with coefficients depending on $y$, with $w=\frac{N}{(N,c^{2})}$. In this expansion, only terms of the form $q_{w}^{l}$ with $l\in\mathbb{Z}+w\alpha$ appear. Moreover, the coefficient $(\mathcal{L}_{N}\varphi)_{c,r,u}$ of $u$ in the function $(\mathcal{L}_{N}\varphi)_{c,r}$ from Equation \eqref{LNdef} has a similar Fourier expansion, but in which only those power of $l$ which are congruent to $\frac{crw}{N}+w\alpha$ modulo $w\mathbb{Z}$ appear, i.e., the expansion is in $q$, containing only terms of the form $q^{n}$ with $n\in\frac{cr}{N}+\alpha+\mathbb{Z}$. \label{LNFour}
\end{cor}

\begin{proof}
A classical argument shows that if $\varphi\in\mathcal{A}_{\kappa}\big(\widetilde{\Gamma}_{1}(N),\rho\big)$ then
$\varphi\big|_{\kappa,\rho}A_{c,d}$ lies in $\mathcal{A}_{\kappa}\big(A_{c,d}^{-1}\widetilde{\Gamma}_{1}(N)A_{c,d},\rho\big)$, and $A_{c,d}^{-1}\widetilde{\Gamma}_{1}(N)A_{c,d}\cap\langle T \rangle=\langle T^{w} \rangle$ by cusp width considerations. Hence $\varphi\big|_{\kappa,\rho}A_{c,d}$ is invariant under $\big|_{\kappa,\rho}T^{w}$, from which we deduce the equality $\varphi_{c,d,u}(\tau+w)=\mathbf{e}(w\alpha)\varphi_{c,d,u}(\tau)$ for the scalar-valued coefficient $\varphi_{c,d,u}$ and any $\tau\in\mathcal{H}$. The form of the Fourier expansion of $\varphi_{c,d,u}$ is now clear. The same assertion thus holds also for $(\mathcal{L}_{N}\varphi)_{c,r}$, which is a linear combination of these functions by Equation \eqref{LNdef}. But Proposition \ref{LNphi} shows that applying $\big|_{\kappa,\rho}T$ (without the $w$\textsuperscript{th} power) to $(\mathcal{L}_{N}\varphi)_{c,r}$ yields $\mathbf{e}\big(\frac{cr}{N}\big)(\mathcal{L}_{N}\varphi)_{c,r}$. The coefficient $(\mathcal{L}_{N}\varphi)_{c,r,u}$ therefore satisfies $(\mathcal{L}_{N}\varphi)_{c,r,u}(\tau+1)=\mathbf{e}\big(\frac{cr}{N}+\alpha\big)\cdot(\mathcal{L}_{N}\varphi)_{c,r,u}(\tau)$. The assertion about the Fourier expansion of $(\mathcal{L}_{N}\varphi)_{c,r,u}$ follows from the fact that the asserted powers of $l$ are precisely those which satisfy $\mathbf{e}\big(\frac{l}{w}\big)=\big(\frac{cr}{N}+\alpha\big)$. Writing $q_{w}^{l}$ as $q^{n}$ for $n=\frac{l}{w}$ now completes the proof of the corollary.
\end{proof}

Assuming that $\varphi\in\mathcal{A}_{\kappa}\big(\widetilde{\Gamma}_{0}(N),\chi\otimes\rho\big)$ for some character $\chi$ modulo $N$, its image under $\mathcal{L}_{N}$ should carry additional structure. To describe it we define, for every $h\in(\mathbb{Z}/N\mathbb{Z})^{\times}$, the map $m_{h} \in Aut(D_{B}(N))$ that multiplies $F$ by $h$ and $\frac{1}{N}E$ by its inverse in $(\mathbb{Z}/N\mathbb{Z})^{\times}$. The map $h \mapsto m_{h}$ embeds $(\mathbb{Z}/N\mathbb{Z})^{\times}$ into $Aut\big(D_{B}(N)\big)$, hence into the automorphism group of spaces of $\mathbb{C}[D_{B}(N)]$-valued functions. We may thus decompose $\mathcal{A}_{\kappa}\big(Mp_{2}(\mathbb{Z}),\rho\otimes\rho_{B}(N)\big)$ as the direct sum of the $\chi$-isotypic subspaces
\begin{equation}
\mathcal{A}_{\kappa}\big(Mp_{2}(\mathbb{Z}),\rho\otimes\rho_{B}(N)\big)^{\chi}=\Big\{
G\in\mathcal{A}_{\kappa}\big(Mp_{2}(\mathbb{Z}),\rho\otimes\rho_{B}(N)\big)\Big|m_{h}(G)=\chi(h)G\ \forall h\in(\mathbb{Z}/N\mathbb{Z})^{\times}\Big\} \label{chiisodef}
\end{equation}
for Dirichlet characters $\chi$ modulo $N$. We can now prove
\begin{prop}
If $\varphi\in\mathcal{A}_{\kappa}\big(\widetilde{\Gamma}_{0}(N),\chi\otimes\rho\big)$ for some Dirichlet character $\chi$ modulo $N$, then
$\mathcal{L}_{N}\varphi$ lies in $\mathcal{A}_{\kappa}\big(Mp_{2}(\mathbb{Z}),\rho\otimes\rho_{B}(N)\big)^{\overline{\chi}}$. \label{LNphichi}
\end{prop}
The complex conjugation on $\chi$ appearing in Proposition \ref{LNphichi}, which
results from our normalization, will turn out more appropriate for the theta
lifts below.

\begin{proof}
Let $h\in(\mathbb{Z}/N\mathbb{Z})^{\times}$, with inverse $g$ modulo $N$, be given. We may replace the summation index $d$ in the formula defining  $(\mathcal{L}_{N}\varphi)_{ch,rg}$ in Equation \eqref{LNdef} by $dh$, which does not affect the condition $(c,d,N)=1$. Note that any element of $\widetilde{\Gamma}_{0}(N)$ lying over a matrix with lower right entry congruent to $h$ modulo $N$ can be taken as our $A_{0,h}$, and left multiplication by such a matrix takes $A_{c,d}$ to $A_{ch,dh}$ (up to Remark \ref{LNinv}). This gives \[(\mathcal{L}_{N}\varphi)_{ch,rg}=\sum_{\substack{d\in\mathbb{Z}/N\mathbb{Z} \\ (c,d,N)=1}}\mathbf{e}\big(-\tfrac{dhrg}{N}\big)\varphi\big|_{\kappa,\rho}A_{ch,dh}=\sum_{\substack{d\in\mathbb{Z}/N\mathbb{Z} \\ (c,d,N)=1}}\mathbf{e}\big(-\tfrac{dr}{N}\big)\varphi\big|_{\kappa,\rho}A_{0,h}A_{c,d}=\chi(h)(\mathcal{L}_{N}\varphi)_{c,r}\] by Equation \eqref{LNdef} and the behavior of $\varphi$ under $A_{0,h}\in\widetilde{\Gamma}_{0}(N)$. From this it now follows, by a simple summation index change in Equation \eqref{LNdef}, that
\[\mathcal{L}_{N}\varphi=\sum_{c,r\in\mathbb{Z}/N\mathbb{Z}}(\mathcal{L}_{N}\varphi)_{ch,rg}\otimes\mathfrak{e}_{ch,rg}=
\sum_{c,r\in\mathbb{Z}/N\mathbb{Z}}\chi(h)\cdot(\mathcal{L}_{N}\varphi)_{c,r}\otimes\mathfrak{e}_{ch,rg}=\chi(h) \cdot m_{h}(\mathcal{L}_{N}\varphi).\] This completes the proof of the proposition.
\end{proof}

\section{A Lattice for $\Gamma_{1}(N)$ \label{L1N}}

In this section, we will recall some facts about a lattice whose discriminant kernel is related to $\Gamma_{1}(N)$. Let $V:=M_{2}(\mathbb{R})_{0}$ be the space of traceless $2\times2$ real matrices, which becomes a quadratic space of signature $(2,1)$ if we set $Q(A):=-\det A$. The induced bilinear pairing is given by $(A,B)=Tr(AB)$. The group $SL_{2}(\mathbb{R})$ operates orthogonally on $V$ by conjugation. This operation defines an isomorphism between $PSL_{2}(\mathbb{R})$ and $SO^{+}(V)$, the connected component of $SO(V)$ containing the identity.

The traceless matrices \[E=\begin{pmatrix} 0 & 1 \\ 0 & 0\end{pmatrix},\quad H=\begin{pmatrix} 1 & 0 \\ 0 & -1\end{pmatrix},\quad\mathrm{and}\quad F=\begin{pmatrix} 0 & 0 \\ 1 & 0\end{pmatrix}\] form a basis for $V$, in which $E$ and $F$ are isotropic and span a hyperbolic plane, and $H$ is orthogonal to both of them and has norm 2. Take $M=\big(\begin{smallmatrix} a & b \\ c & d\end{smallmatrix}\big) \in SL_{2}(\mathbb{R})$, with inverse $M^{-1}=\big(\begin{smallmatrix} d & -b \\ -c & a\end{smallmatrix}\big)$. Conjugating $E$, $H$, and $F$ yields
\begin{equation} \begin{pmatrix} -ac & a^{2} \\ -c^{2} & ac\end{pmatrix},\quad\begin{pmatrix} ad+bc & -2ab \\ 2cd & -ad-bc\end{pmatrix},\quad\mathrm{and}\quad\begin{pmatrix} bd & -b^{2} \\ d^{2} & -bd\end{pmatrix} \label{SL2RM2R0}
\end{equation}
respectively, so that the action of $M$ in the basis $-E$, $\frac{1}{2}H$, and $F$ is represented by the matrix classically known as $\mathrm{Sym}^{2}M$.

For a lattice $L \subseteq V$ of full rank, denote $SAut^{+}(L):=Aut(L) \cap SO^{+}(V)$. Given $N\in\mathbb{N}$, we define the lattice
\[L_{1}(N):=\mathbb{Z} \cdot H\oplus\mathbb{Z} \cdot E\oplus\mathbb{Z} \cdot NF
\cong L_{1} \oplus L_{B}(N),\] where $L_{1}\cong\mathbb{Z} \cdot H$ is the 1-dimensional lattice giving rise to the Weil representation $\rho_{1}$ from Equation \eqref{rho1TSR} and \eqref{rho1Gamma04} as well as the theta functions from Equation \eqref{Theta} . The dual lattice $L_{1}^{*}(N)\cong\mathbb{Z}\cdot\frac{1}{2}H \oplus L_{B}^{*}(N)$ is spanned by $\frac{1}{2}H$, $\frac{1}{N}E$, and $F$, and $D_{1}(N) \cong D_{1} \oplus D_{B}(N)$. The Weil representation $\rho_{L_{1}(N)}$, which we denote by $\rho_{1}(N)$, is thus isomorphic to $\rho_{1}\otimes\rho_{B}(N)$.

The full group $SAut^{+}\big(L_{1}(N)\big)$ is described in Remark \ref{SAutL1N} below. We shall only make use of the following proposition.

\begin{prop}
The group $\Gamma_{0}(N)$ preserves the lattice $L_{1}(N)$. The matrix $\big(\begin{smallmatrix} a & b \\ c & d\end{smallmatrix}\big)\in\Gamma_{0}(N)$ acts trivially on the summand $D_{1}(N)$ and as $m_{d^{2}} \in Aut\big(D_{B}(N)\big)$ defined before Equation \eqref{chiisodef} on $D_{B}(N)$. \label{Gamma0NL1N}
\end{prop}

\begin{rmk}
As a consequence, the discriminant kernel of $L_{1}(N)$ contains all matrices $\big(\begin{smallmatrix} a & b \\ c & d\end{smallmatrix}\big) \in \Gamma_{0}(N)$ satisfying $d^2 \equiv 1 \bmod{N}$. We denote this subgroup of $\Gamma_{0}(N)$, which in particular contains $\Gamma_{1}(N)$, by $\Gamma_{1}^{\sqrt{1}}(N)$. \label{disckerL1N}
\end{rmk}

\begin{proof}
It is clear from Equation \eqref{SL2RM2R0} that if $a$, $b$, $c$, and $d$ are integers with $N|c$ then the images of the basis $E$, $H$, and $-NF$ of $L_{1}(N)$ under $\big(\begin{smallmatrix} a & b \\ c & d\end{smallmatrix}\big)$ also lie in $L_{1}(N)$. Moreover, using the condition $ad-bc=1$ we find that $\frac{1}{2}H+L_{1}(N)$ is fixed, while the images of the other generators $\frac{1}{N}E+L_{1}(N)$ and $F+L_{1}(N)$ of $D_{1}^{}(N)$ equal $\frac{a^{2}}{N}E+L_{1}(N)$ and $d^{2}F+L_{1}(N)$ respectively. This proves the proposition.
\end{proof}

\begin{rmk}
Let $\Gamma_{0}^{*}(N)$ be the group which is obtained by adding to $\Gamma_{0}(N)$ the \emph{Atkin--Lehner involutions}, i.e., matrices of the form $M=\big(\begin{smallmatrix} e\sqrt{\mu} & f/\sqrt{\mu} \\ g\sqrt{\mu}N/\mu & h\sqrt{\mu}\end{smallmatrix}\big)$, with a divisor $\mu$ of $N$ which satisfies $\big(\mu,\frac{N}{\mu}\big)=1$ and integers $e$, $f$, $g$, and $h$ satisfying $eh\mu-fg\frac{N}{\mu}=1$. It contains $\Gamma_{0}(N)$ as a normal subgroup, and the quotient is isomorphic to a product of $\{\pm1\}$'s, one for each prime divisor of $N$. The action of that larger group also preserves $L_{1}(N)$. It leaves the generator $\frac{1}{2}H$ invariant as well, but only elements with $\mu=1$ (i.e., from $\Gamma_{0}(N)$) preserve the subgroups $(\mathbb{Z}/N\mathbb{Z})\frac{1}{N}E$ and $(\mathbb{Z}/N\mathbb{Z})F$. One can show that $\Gamma_{0}^{*}(N)$ is the stabilizer of $\frac{1}{2}H$ in $SAut^{+}\big(L_{1}(N)\big)$ (so that the group $\Gamma_{1}^{\sqrt{1}}(N)$ from Proposition \ref{Gamma0NL1N} is the full discriminant kernel of $L_{1}(N)$). Moreover, if $4|N$ then $\Gamma_{0}^{*}(N)$ is the full group $SAut^{+}\big(L_{1}(N)\big)$, while otherwise it has index 4 in that group. All this can be proved as in the forthcoming paper \cite{[Ze3]} of the second author. The proof from \cite{[Ze3]} also covers the assertion from Proposition 2.2 of \cite{[BO]}, concerning a lattice whose discriminant kernel is $\Gamma_{0}(N)$ (and whose the full $SAut^{+}$ group is $\Gamma_{0}^{*}(N)$ for any $N$). \label{SAutL1N}
\end{rmk}

\section{Theta Lifts \label{ThetaLift}}

To obtain the main result of this paper, we shall apply Theorem 14.3 of \cite{[B]}. We will introduce the notations and definitions sufficient to give a full explanation relevant to our case. For more details, the reader is referred to Sections 13 and 14 of \cite{[B]}, as well as Section 3.2 of \cite{[Br1]} or Section 2.2 of \cite{[Ze2]}.

Let $L$ be a lattice of signature $(2,b_{-})$ containing a primitive isotropic
vector $z$, define $K:=(L \cap z^{\perp})/\mathbb{Z}z$, and take $\zeta \in L^{*}$ with $(\zeta,z)=1$. We define $C$ to be the positive cone of positive norm vectors in the Lorentzian vector space $K_{\mathbb{R}}$ according to the choice of $z$ and an orientation on the positive definite parts of $L$, and use $W$ to denote a Weyl chamber in $C$. Let $M\in\mathbb{N}$ be such that  $(z,L)=M\mathbb{Z}$, and denote the $k$\textsuperscript{th} Bernoulli polynomial by $B_{k}$ for any $k\in\mathbb{N}$. Given $F\in\mathcal{M}_{\kappa}^{!}(Mp_{2}(\mathbb{Z}),\rho_{L})$, with Fourier expansion $\sum_{\gamma,n}c_{\gamma,n}q^{n}\mathfrak{e}_{\gamma}$, we denote $F_{K}$ the modular form of representation $\rho_{K}$ which is obtained from $F$ by the operator, denoted  $\downarrow$ in \cite{[Br2]} and others, which is associated with the isotropic subgroup of $D_{L}$ that is generated by a primitive element of $\mathbb{Q}z \cap L^{*}$. We can now quote Theorem 14.3 of \cite{[B]}:
\begin{thm}
In the notations above with $F\in\mathcal{M}_{\kappa}(Mp_{2}(\mathbb{Z}),\rho_{L})$ and $\kappa=1-\frac{b_{-}}{2}+k$ for some $2 \leq k\in\mathbb{N}$, the following Fourier expansion defines a holomorphic function in the variable $Z \in K_{\mathbb{R}}+iC$, which is an automorphic form of weight $k$ with respect to the discriminant kernel $\Gamma_{L}$:
\[\frac{M^{k-1}}{2k}\sum_{\varepsilon=1}^{M}\sum_{\delta\in\mathbb{Z}/M\mathbb{Z}}B_{k}\bigg(\frac{\varepsilon}{M}\bigg)
\mathbf{e}\bigg(\frac{\varepsilon\delta}{M}\bigg)c_{\frac{\delta z}{M}+L,0}+\sum_{n=1}^{\infty}\sum_{\substack{\lambda \in K^{*} \\ (\lambda,W)>0}}n^{k-1}\mathbf{e}\big(n(\lambda,Z)\big)\sum_{\substack{\eta \in L^{*}/L \\ \eta|_{z^{\perp}}=\pi^{*}\lambda}}\mathbf{e}\big(n(\eta,\zeta)\big)c_{\eta,\frac{\lambda^{2}}{2}}.\]
When $F(\tau)\in\mathcal{M}_{\kappa}^{!}\big(Mp_{2}(\mathbb{Z}),\rho_{L}\big)$, the resulting automorphic form is meromorphic, with poles of order $k$ along special divisors on $K_{\mathbb{R}}+iC$. Both statements extend to the case $k=1$, provided that one adds a theta integral involving $F_{K}$ to the Fourier expansion. If $F$ is a harmonic weak Maass form whose image under the operator $\xi$ from \cite{[BF]} is a cusp form, then completing this function by a certain sum which is based on negative norm vectors of $K$ still yields an automorphic form of weight $k$, with the same singularities, which is harmonic with respect to the weight $k$ Laplacian on $K_{\mathbb{R}}+iC$. \label{Bor143}
\end{thm}
The last assertion of Theorem \ref{Bor143} is not contained directly in Theorem 14.3 of \cite{[B]}, but follows from the results of \cite{[Br1]} and \cite{[Ze2]}. We remark that the exact form of the poles from Theorem \ref{Bor143} can be described using the Fourier coefficients of negative indices of the lifted modular form $F$, so that one can say more about the poles in Theorems \ref{main}, \ref{Stf}, and \ref{level} and Proposition \ref{mainchar} below. We leave these details for further investigation.

\smallskip

Before we present the main result of this paper, we need to introduce a little more notation. Given $\varepsilon\in\{\pm1\}$, $N\in\mathbb{N}$, $d\in(\mathbb{Z}/N\mathbb{Z})^{\times}$, and  $f\in\mathcal{A}_{k+1/2}(\widetilde{\Gamma}_{N},\psi^{\varepsilon})$, we denote $c_{n}^{(d)}$ the $n$\textsuperscript{th} Fourier coefficient of the image of $f$ under the diamond operator $\langle d \rangle_{\varepsilon}$ defined by
\begin{equation}
\langle d \rangle_{\varepsilon}f:=\mathcal{L}_{\varepsilon}^{-1}\big((\mathcal{L}_{\varepsilon}f)\big|_{k+1/2,\rho_{\varepsilon}}A_{0,d}\big)
\in\mathcal{A}_{k+1/2}^{+}(\widetilde{\Gamma}_{N},\psi^{\varepsilon}), \label{diamond}
\end{equation}
where $A_{0,d} \in Mp_{2}(\mathbb{Z})$ is defined in Proposition \ref{LNphi}. These Fourier coefficients are functions of $y$ in general, unless $f \in \mathcal{M}_{k+1/2}^{!,+}(\widetilde{\Gamma}_{N},\psi^{\varepsilon})$, where they are constants. In correspondence with Remark \ref{L-}, we shall need only the definition using $\mathcal{L}_{+}$ in Equation \eqref{diamond} for the results of this Section. However, the next Section (in particular the proof of Theorem \ref{Stf}) will use Equation \eqref{diamond} with the operator $\mathcal{L}_{-}$, as well as the following
\begin{prop}
The function $\langle d \rangle_{\varepsilon}f$ coincides with $f\big|_{k+1/2,\psi^{\varepsilon}}A$ for any $A\in\widetilde{\Gamma}_{0}(16N)$ having lower right entry in $d+N\mathbb{Z}$. This extends the definition of $\langle d \rangle_{\varepsilon}$ to all of $f\in\mathcal{A}_{k+1/2}(\widetilde{\Gamma}_{N},\psi^{\varepsilon})$. In the case where $4|N$, the two spaces $\mathcal{A}_{k+1/2}(\widetilde{\Gamma}_{N},\psi^{\pm1})$ coincide by Proposition \ref{char4notdivN}, and the two functions $\langle d \rangle_{+}f$ and $\langle d \rangle_{-}f$ differ by a multiplicative constant $\big(\frac{-1}{d}\big)$. \label{diamSV}
\end{prop}

\begin{proof}
We may choose our representative for $A_{0,d}$ in Equation \eqref{diamond} from the group $\widetilde{\Gamma}_{0}(4N)\cap\widetilde{\Gamma}^{0}(4)$. In this case $\rho_{\varepsilon}(A_{0,d})$ is just the scalar $\psi^{\varepsilon}(A_{0,d})$ by Equation \eqref{rho1Gamma04}, so that we may replace the vector-valued operator $\big|_{k+1/2,\rho_{\varepsilon}}A_{0,d}$ by the scalar-valued one $\big|_{k+1/2,\psi^{\varepsilon}}A_{0,d}$. The proof of Theorem \ref{rho1lift} now shows that $\langle d \rangle_{\varepsilon}f$ can be written just as $f\big|_{k+1/2,\psi^{\varepsilon}}A$ where $A=A_{0}A_{0,d}A_{0}^{-1}$ lies in $\widetilde{\Gamma}_{0}(16N)$ and has lower right entry in $d+N\mathbb{Z}$. Moreover, the slash operator associated with any element of that group may be obtained in this way, which proves the first assertion. The extendability is now immediate. Assuming that $4|N$, so that the residue of the odd number $d$ is well-defined modulo 4, the two operators $\big|_{k+1/2,\psi^{\pm1}}A$ with $A\in\widetilde{\Gamma}_{0}(16N)$ as above differ by the factor $\psi^{2}(A)=\big(\frac{-1}{d}\big)$. This completes the proof of the proposition.
\end{proof}

The main result will be more conveniently written in terms of the \emph{partial zeta functions} defined by $\zeta_{N}^{(d)}(s)=\sum_{n \equiv d(\mathrm{mod\ }N)}\frac{1}{n^{s}}$ (with $n$ positive), where $(d,N)=1$. We can now prove the following
\begin{thm}
Given $f\in\mathcal{M}_{k+1/2}^{+}(\widetilde{\Gamma}_{N},\psi)$ with $k\in\mathbb{N}$, the function $\mathcal{S}_{1}^{(N)}f:\mathcal{H}\to\mathbb{C}$ defined by
\[\mathcal{S}_{1}^{(N)}f(\sigma):=-\sum_{\substack{d=1 \\ (d,N)=1}}^{N}\frac{\zeta_{N}^{(d)}(1-k)c_{0}^{(d)}}{2}+\sum_{l=1}^{\infty}\Bigg[\sum_{\substack{d|l \\ (d,N)=1}}d^{k-1}c_{\frac{l^{2}}{d^{2}}}^{(d)}\Bigg]\mathbf{e}(l\sigma)\] lies in $\mathcal{M}_{2k}\big(\Gamma_{1}^{\sqrt{1}}(N)\big)$. If $f\in\mathcal{M}_{k+1/2}^{!,+}(\widetilde{\Gamma}_{N},\psi)$, then $\mathcal{S}_{1}^{(N)}(f)$ is a meromorphic modular form, with poles of order $k$ at CM points. \label{main}
\end{thm}

\begin{proof}
Take $L=L_{1}(N)$ in Theorem \ref{Bor143}, with $z=E$ and $\zeta=F$. Hence
$b_{-}=1$, $K$ is $\mathbb{Z}H$ with its dual spanned by $\frac{1}{2}H$, and
$M=N$. The variety $K_{\mathbb{R}}+iC$ is just $\mathcal{H}$ (with no Weyl
chambers), and the pairing of $\frac{m}{2}H$ with the element representing
$\sigma\in\mathcal{H}$ is just $m\sigma$. If we now assume that $f\in\mathcal{M}_{k+1/2}^{+,!}(\widetilde{\Gamma}_{N},\psi)$, then $\mathcal{L}_{N}\mathcal{L}_{+}f\in\mathcal{M}_{k+1/2}^{!}\big(Mp_{2}(\mathbb{Z}),\rho_{1}\otimes\rho_{B}(N)\big)$ by Theorem \ref{rho1lift} and Proposition \ref{LNphi}. As $\rho_{1}(N) \cong \rho_{1} \otimes \rho_B(N)$, we may apply Theorem \ref{Bor143} to $(\mathcal{L}_{N}\mathcal{L}_{+}f)/N$. An automorphic form of weight $k$ on a 1-dimensional Grassmannian is the same as an automorphic form of weight $2k$ on $\mathcal{H}$, and Remark \ref{disckerL1N} shows that the discriminant kernel contains $\Gamma_{1}^{\sqrt{1}}(N)$ (in fact, the latter group is the full discriminant kernel by Remark \ref{SAutL1N}).

Thus, we need to substitute the Fourier coefficients of
$(\mathcal{L}_{N}\mathcal{L}_{+}f)/N$ into the expressions from Theorem
\ref{Bor143}. Note that for $\lambda=\frac{m}{2}H$ we have $\frac{\lambda^{2}}{2}=\frac{m^{2}}{4}$, and the elements $\eta \in
L^{*}$ which restrict to $\pi^{*}\lambda$ on $z^{\perp}$ are those of the form
$\frac{m}{2}H+\frac{r}{N}E$, where considering these elements modulo $L$ means
that $r$ is taken in $\mathbb{Z}/N\mathbb{Z}$. Now, the Fourier coefficients having index $\frac{m}{2}H+\frac{r}{N}E \in D_{1}(N)$ are $\frac{1}{N}$ times the Fourier coefficients of the function denoted $(\mathcal{L}_{N}\mathcal{L}_{+}f)_{0,r,u}$ in the terminology of Corollary \ref{LNFour}, where $u$ is $\mathfrak{e}_{\delta}$ with $\delta$ being the element of $\{0,1\}$ which is congruent to $m$ modulo 2. Moreover, Equation \eqref{LNdef} or the proof of Corollary \ref{LNFour} shows that these Fourier coefficients are the Fourier coefficients of the scalar-valued function $\frac{1}{N}\sum_{d}\mathbf{e}\big(-\frac{dr}{N}\big)(\mathcal{L}_{+}f)_{0,d,u}$ where $d$ runs over the elements of $(\mathbb{Z}/N\mathbb{Z})^{\times}$. Note that the condition $n\in\frac{cr}{N}+\alpha+\mathbb{Z}$ from Corollary \ref{LNFour} is satisfied for $n=\frac{m^{2}}{4}$, since $c=0$ and the eigenvalue $\mathbf{e}(\alpha)$ is $\mathbf{e}\big(\frac{\delta}{4}\big)=\mathbf{e}\big(\frac{m^{2}}{4}\big)$. Hence no vanishing is implied by this corollary. Equation \eqref{diamond} (with $\varepsilon=+$) now implies that the $d$\textsuperscript{th} summand here is $\mathbf{e}\big(-\frac{dr}{N}\big)$ times the coefficient of $\mathfrak{e}_{\delta}$ in $\mathcal{L}_{+}(\langle d \rangle_{+}f)$. The definition of $\mathcal{L}_{+}$ in Theorem \ref{rho1lift} and Equation \eqref{fjdef} now shows that the Fourier coefficient of $q^{m^{2}/4}$ in the part of $\mathcal{L}_{+}(\langle d \rangle_{+}f)$ appearing in front of $\mathfrak{e}_{\delta}$ is just $c_{m^{2}}^{(d)}$ (the index $\frac{m^{2}}{4}$ changes to $m^{2}$ because of the 4 multiplying the argument $\tau$ in the translation formula). The second summand in the expression from Theorem \ref{Bor143} then becomes
\[\sum_{n=1}^{\infty}n^{k-1}\sum_{m=1}^{\infty}\mathbf{e}(mn\sigma)\sum_{r\in\mathbb{Z}/N\mathbb{Z}}
\mathbf{e}\bigg(\frac{nr}{N}\bigg)\sum_{\substack{d\in\mathbb{Z}/N\mathbb{Z} \\ (d,N)=1}}\mathbf{e}\bigg(\frac{-dr}{N}\bigg)\frac{c_{m^{2}}^{(d)}}{N}\]
Summing over $r\in\mathbb{Z}/N\mathbb{Z}$ (to get $n=d$) and substituting $l=mn$ then give us the desired second term.

The constant term involves Fourier coefficients with index $\frac{\delta}{N}E$, so that a similar analysis shows that this constant term equals
\[\frac{N^{k-1}}{2k}\sum_{\varepsilon=1}^{N}B_{k}\bigg(\frac{\varepsilon}{N}\bigg)\sum_{\delta\in\mathbb{Z}/N\mathbb{Z}}
\mathbf{e}\bigg(\frac{\varepsilon\delta}{N}\bigg)\sum_{\substack{d\in\mathbb{Z}/N\mathbb{Z} \\ (d,N)=1}}\mathbf{e}\bigg(\frac{-d\delta}{N}\bigg)\frac{c_{0}^{(d)}}{N}.\] We sum over $\delta$, so that only terms with $\varepsilon=d$ remain, and use Equation (11) on page 27 of \cite{[EMOT]} to replace $N^{k-1}B_{k}\big(\frac{d}{N}\big)/k$ by $-\zeta_{N}^{(d)}(1-k)$. This yields the required constant term for $k\geq2$. When $k=1$, note that $K$ is the lattice $L_{1}$ whose Weil representation is $\rho_{1}$, and the isotropic subgroup is generated by $\frac{1}{N}z$. Thus \[(\mathcal{L}_{N}\mathcal{L}_{+}f)_{K}=\sum_{r\in\mathbb{Z}/N\mathbb{Z}}(\mathcal{L}_{+}f)_{0,r}=
\sum_{r\in\mathbb{Z}/N\mathbb{Z}}\sum_{d\in(\mathbb{Z}/N\mathbb{Z})^{\times}}\mathbf{e}\bigg(-\frac{dr}{N}\bigg)\mathcal{L}_{+}f\big|_{3/2,\rho_{1}}A_{0,d}= \sum_{d,r}\mathbf{e}\bigg(-\frac{dr}{N}\bigg)\mathcal{L}_{+}(\langle d \rangle_{+}f).\] But the coefficient of $\mathcal{L}_{+}(\langle d \rangle_{+}f)$ here vanishes unless $N|d$, which implies the vanishing of $(\mathcal{L}_{N}\mathcal{L}_{+}f)_{K}$ unless $N=d=1$, since the conditions $N|d$ and $d\in(\mathbb{Z}/N\mathbb{Z})^{\times}$ are contradictory. On the other hand, in the remaining case the function $(\mathcal{L}_{N}\mathcal{L}_{+}f)_{K}$ lies in $\mathcal{A}_{3/2}(Mp_{2}(\mathbb{Z}),\rho_{1})$, a space which is trivial by Remark \ref{VVtriv}. Thus the constant term remains the same also for $k=1$. The assertion now follows from Theorem \ref{Bor143}, since the holomorphicity there is also at the cusps, and since the special divisors on $\mathcal{H}$ which arise from our lattice $L_{1}(N)$ are precisely the CM points on $\mathcal{H}$. This proves the theorem.
\end{proof}

\begin{rmk}
The last assertion of Theorem \ref{Bor143} allows us to extend Theorem \ref{main} also to the case where $f$ is a harmonic weak Maass form of weight $k+\frac{1}{2}$ such that $\xi_{k+1/2}f$ is a cusp form. Indeed, the completion is based on negative norm elements of $K^{*}$, but in our case $K$ is positive definite. Hence no such completion appears, and only the (meromorphic) expression which is based on the positive part of $f$ remains (such an argument already appears, for the case of the classical lift from weight $\frac{1}{2}$ to weight 0, in \cite{[BO]}). As the weight $\frac{3}{2}-k$ of $\xi_{k+1/2}f$ is negative for $k\geq2$, and there are no cusp forms of negative weight, there are no such harmonic weak Maass forms with non-trivial cuspidal $\xi$-image. This extension is therefore non-trivial only for $k=1$. The same assertion holds also for Propositions \ref{mainchar} and \ref{highlev} and Theorems \ref{Stf} and \ref{level} below. \label{Maass}
\end{rmk}

\begin{rmk}
If $f$ is a cusp form then so is $\mathcal{L}_{N}\mathcal{L}_{+}f$, and $\mathcal{S}_{1}^{(N)}f$ has no constant term at the cusp $\infty$. This, however, does not imply that $\mathcal{S}_{1}^{(N)}f$ is a cusp form, since non-zero constant terms may appear at the other cusps. One can show, using results relating the growth of the Fourier coefficients at $\infty$ and vanishing at the other cusps, that the $\mathcal{S}_{1}^{(N)}$-image of a cusp form is a cusp form wherever $k\geq2$. On the other hand, if $k=1$ then elements of the space spanned by unary theta series of weight $\frac{3}{2}$ are sent to Eisenstein series, while the $\mathcal{S}_{1}^{(N)}$-image of the orthogonal complement of this space is contained in the space of cusp forms (see \cite{[C1]}, and also \cite{[P]} for a decomposition algorithm). \label{cusptocusp}
\end{rmk}

\smallskip

Recalling the decomposition of $\mathcal{A}_{k+1/2}(\widetilde{\Gamma}_{N},\psi)$ (and its
plus-subspace) according to characters as in Proposition \ref{char4notdivN} (which clearly preserves holomorphicity and weak holomorphicity), we also obtain the following proposition.
\begin{prop}
Given $f(\tau)=\sum_{n=0}^{\infty}c_{n}q^{n}\in\mathcal{M}_{k+1/2}^{+}(\widetilde{\Gamma}_{0}(4N),\psi^{2k+1}\omega_{\chi})$ for some Dirichlet character $\chi$ modulo $N$ satisfying $\chi(-1)=(-1)^{k}$, we have
\begin{equation}
\mathcal{S}_{1}^{(N)}f(\sigma)=-\frac{L(1-k,\chi)}{2}c_{0}+\sum_{l=1}^{\infty}\Bigg[\sum_{\substack{d|l \\ (d,N)=1}}d^{k-1}\chi(d)c_{\frac{l^{2}}{d^{2}}}\Bigg]\mathbf{e}(l\sigma)\in\mathcal{M}_{2k}\big(\Gamma_{0}(N),\chi^{2}\big).
\end{equation}
In case the modular form $f$ is in $\mathcal{M}_{k+1/2}^{!,+}(\widetilde{\Gamma}_{0}(4N),\psi^{2k+1}\omega_{\chi})$, its lift $\mathcal{S}_{1}^{(N)}f$ is described by the same Fourier expansion, and it still satisfies the modularity condition with character $\chi^{2}$ with respect to $\Gamma_{0}(N)$, but it is now meromorphic with poles of order $k$ at CM points. \label{mainchar}
\end{prop}
Indeed, the group $\Gamma_{1}^{\sqrt{1}}(N)$ appearing in Theorem \ref{main} is
the intersection of the kernels of all the squares of Dirichlet characters
modulo $N$. We emphasize that the Dirichlet $L$-function appearing in the
constant term is associated with $\chi$ as a character modulo $N$, not with its primitive version.

\begin{proof}
To prove that $\mathcal{S}_{1}^{(N)}(f)\in\mathcal{M}_{2k}\big(\Gamma_{0}(N),\chi^{2}\big)$, recall that $\mathcal{S}_{1}^{(N)}(f)$ is given by the (regularized) pairing of $\mathcal{L}_{N}\mathcal{L}_{+}f/N$ with a certain vector-valued theta function associated to $L_{1}(N)$. The modularity of $\mathcal{S}_{1}^{(N)}(f)$ follows from the invariance of this theta function under the discriminant kernel. A more general element of $SAut^{+}\big(L_{1}(N)\big)$ operates on this theta function through permuting its components according to the appropriate automorphism of $D_{1}(N)$. The orthonormality of the canonical basis of $\mathbb{C}\big[D_{1}(N)\big]$ shows that the resulting function coincides with the theta lift of the element of $\mathcal{M}_{k+1/2}\big(Mp_{2}(\mathbb{Z}),\rho_{1}(N)\big)$ obtained from $\mathcal{L}_{N}\mathcal{L}_{+}f/N$ by permuting the components according to the inverse automorphism of $D_{1}(N)$.

Now, consider the action of $\big(\begin{smallmatrix} a & b \\ c & d\end{smallmatrix}\big)\in\Gamma_{0}(N)$. Proposition \ref{Gamma0NL1N} shows that its operation on $D_{1}(N)$, hence on the theta function, is via the automorphism associated with $d^{2}\in(\mathbb{Z}/N\mathbb{Z})^{\times}$. Hence we must consider $\mathcal{L}_{N}\mathcal{L}_{+}f/N$ twisted by the automorphism $d^{-2}$. But $\mathcal{L}_{+}f$ lies in $\mathcal{M}_{k+1/2}\big(\Gamma_{0}(N),\chi\otimes\rho_{1}\big)$ by Theorem \ref{rho1lift}, and it follows from Proposition \ref{LNphichi} that the automorphism $d^{-2}$ of $D_{1}(N)$ multiplies $\mathcal{L}_{N}\mathcal{L}_{+}f$ by $\overline{\chi}(d^{-2})=\chi^{2}(d)$. This proves the modularity with respect to $\Gamma_{0}(N)$. For the Fourier expansion, Equation \eqref{diamond} and the fact that $\mathcal{L}_{+}f\in\mathcal{M}_{k+1/2}\big(\Gamma_{0}(N),\chi\otimes\rho_{1}\big)$ show that $\langle d \rangle_{+}f=\chi(d)f$. Hence every superscript $(d)$ can be replaced by the multiplier $\chi(d)$. Theorem \ref{main} and the relations between Dirichlet $L$-functions and partial zeta functions now complete the proof of the proposition.
\end{proof}

\begin{rmk}
Assuming that $\chi$ is quadratic in Proposition \ref{mainchar}, one may consider the action of the group $\widetilde{\Gamma}_{0}^{*}(4N)$, including the Atkin--Lehner involutions (of half-integral weight), on our lifted modular form $f$. Remark \ref{SAutL1N} states that the Atkin--Lehner involutions operate on $\mathcal{S}_{1}^{(N)}f$ also via automorphisms of the components of the theta kernel. It would be interesting to see whether Atkin--Lehner eigenfunctions are mapped via $\mathcal{S}_{1}^{(N)}f$ to Atkin--Lehner eigenfunctions. However, as our proof uses Weil representations of $Mp_{2}(\mathbb{Z})$, which have no natural extension to $\Gamma_{0}^{*}(N)$, it seems that investigating this question requires tools additional to those developed in this paper. \label{ALprob}
\end{rmk}

\section{Higher Indices and Levels \label{HiLev}}

The lift $\mathcal{S}_{1}^{(N)}$ from Theorem \ref{main} and Proposition
\ref{mainchar} involves only Fourier coefficients of square indices. For the coefficients in the other square classes, we shall use the properties of functions obtained by rescaling the variable of modular forms by an integer. The modularity properties of these functions are simple and well-known in the integral weight case. On the other hand, things are a little more delicate in the half-integral case.

For $t\in\mathbb{N}$ we denote the character $d\mapsto\big(\frac{t}{d}\big)$ by $\chi_{t}$, and let $t_{2}$ be the largest positive odd divisor of $t$ as well as $v_{2}(t)$ the 2-adic valuation of $t$. Note that $\chi_{t}$ is always an even character, since it decomposes as a power of the even character $\chi_{2}$ times $\chi_{t_{2}}$, and we have $\chi_{t_{2}}(-1)=(-1)^{\varepsilon(t_{2})}\big(\frac{-1}{t_{2}}\big)=1$ by the quadratic reciprocity law. Here and below, $\varepsilon(h)$ stands for the image of $\frac{h-1}{2}$ in $\mathbb{Z}/2\mathbb{Z}$ for any odd number $h$. Note that as in \cite{[Str]} and others, in case $t$ is odd we allow $d$ to be even, by defining $\big(\frac{t}{2}\big)$ to be the same as $\big(\frac{2}{t}\big)$. Given such $t$ and a function $f:\mathcal{H}\to\mathbb{C}$, we define
\begin{equation}
f_{t}(\tau):=f(t\tau)=t^{-\kappa/2}\big(f\big|_{\kappa}C_{t}\big)(\tau),\quad\mathrm{where}\quad
C_{t}:=\Bigg(\begin{pmatrix} \sqrt{t} & 0 \\ 0 & 1/\sqrt{t}\end{pmatrix},\frac{1}{\sqrt[4]{t}}\Bigg) \in Mp_{2}(\mathbb{R})\quad\mathrm{and}\quad \kappa\in\tfrac{1}{2}\mathbb{Z}. \label{ftCtdef}
\end{equation}
Here we start with the following lemma.
\begin{lem}
Suppose that $f\in\mathcal{A}_{k+1/2}(\widetilde{\Gamma}_{N},\psi^{2k+\xi})$ for some $\xi\in\{\pm1\}$, and $t\in\mathbb{N}$. Then $f_{t}$ lies in the space $\mathcal{A}_{k+1/2}\big(\widetilde{\Gamma}_{0}(4Nt)\cap\widetilde{\Gamma}_{N},\psi^{2k+\xi}\chi_{t}\big)$. This implies that $f_{t} \in \mathcal{A}_{k+1/2}(\widetilde{\Gamma}_{Nt},\psi^{2k+\upsilon})$ for some $\upsilon\in\{\pm1\}$ if either $4|N$, $4|t$, or $t$ is odd and
$\big(\frac{-1}{t}\big)=\xi\upsilon$, but not otherwise. In case $f\in\mathcal{A}_{k+1/2}\big(\widetilde{\Gamma}_{0}(4N),\psi^{2k+1}\omega_{\chi}\big)$ for some character $\chi$ modulo $N$ with $\chi(-1)=\xi$ and at least one of the conditions above is satisfied, we get that $f_{t}$ lies in $\mathcal{A}_{k+1/2}\big(\widetilde{\Gamma}_{0}(4Nt),\psi^{2k+1}\omega_{\eta}\big)$, where $\eta(d):=\chi(d)\big(\frac{\xi\upsilon t}{d}\big)$ defines a character modulo $Nt$ with $\eta(-1)=\upsilon$. \label{relfft}
\end{lem}
It is clear that if $f$ is (weakly) holomorphic then so is $f_{t}$, so that we
may replace every $\mathcal{A}$ by $\mathcal{M}$ or by $\mathcal{M}^{!}$ in Lemma \ref{relfft} and still get a valid assertion. The same statement extends to harmonic weak Maass forms, with restrictions on their $\xi$-images, as in Remark \ref{Maass}.

\begin{proof}
If $A\in\widetilde{\Gamma}_{0}(4Nt)$ lies over a matrix $\big(\begin{smallmatrix} a & b \\ c & d\end{smallmatrix}\big)\in\Gamma_{0}(4Nt)$ then $A_{t}:=C_{t}A_{t}C_{t}^{-1}$ belongs to $\widetilde{\Gamma}_{0}(4N)$ and lies over $\big(\begin{smallmatrix} a & bt \\ c/t & d\end{smallmatrix}\big)$. It follows from Equation \eqref{ftCtdef} and the definition of $A_{t}$ that $(f_{t}\big|_{k+1/2}A)(\tau)=\big(f\big|_{k+1/2}A_{t}\big)(t\tau)$. In case $A$ lies also in $\widetilde{\Gamma}_{N}$, so is $A_{t}$, and $f\big|_{k+1/2}A_{t}$ equals $\psi^{2k+\xi}(A_{t})f$. Comparing the expressions for $\psi(A)$ and $\psi(A_{t})$ appearing in Equation \eqref{rho1Gamma04}, we find that $d$ and the metaplectic sign remain the same, but $c$ is replaced by $\frac{c}{t}$. It follows that  $\psi^{2k+\xi}(A_{t})$ can be written as $\psi^{2k+\xi}(A)\chi_{t}(d)$. This proves the first assertion.

In particular, $f_{t}$ lies in $\mathcal{A}_{k+1/2}(\widetilde{\Gamma}_{Nt},\psi^{2k+\xi}\chi_{t})$.
We thus have to find under which condition the characters $\psi^{2k+\xi}\chi_{t}$ and $\psi^{2k+\upsilon}$ coincide on $\widetilde{\Gamma}_{Nt}$. Now, for $A\in\widetilde{\Gamma}_{Nt}$ as above we know that $d$ is odd and satisfies $d\equiv1(\mathrm{mod\ }Nt)$. The character $\chi_{t}$ decomposes as $\chi_{2}^{v_{2}(t)}\chi_{t_{2}}$ where the first multiplier has conductor 8 if $v_{2}(t)$ is odd, and the conductor of the second one divides $4t_{2}$. Hence the total conductor divides $8t_{2}$. Now, if $4|N$ then the two (odd) powers of $\psi$ coincide on $\widetilde{\Gamma}_{N}$ hence also on $\widetilde{\Gamma}_{Nt}$. In addition, we have $4t_{2}|Nt$, and even $8t_{2}|Nt$ if $t$ is even. The conductor of $\chi_{t}$ thus divides $Nt$ in any such case, proving the assertion if $4|N$. If $4|t$ then the conductor of $\chi_{t}$ divides $t$. To see this, observe that if $\frac{t}{4}$ is odd then the conductor divides $4t_{2}|t$, while otherwise it divides $8t_{2}|t$. As the two powers of $\psi$ coincide on $\widetilde{\Gamma}_{Nt}$ also in this case, the assertion follows here as well.

We now assume that $t$ is odd, and that $4 \nmid N$. We then apply the quadratic reciprocity law to find that $\chi_{t}(d)$ equals $(-1)^{\varepsilon(d)\varepsilon(t)}\big(\frac{d}{t}\big)$, where the latter multiplier equals 1 since $d\equiv1(\mathrm{mod\ }Nt)$. Now, $(-1)^{\varepsilon(d)}$ is the same as $\psi^{2}(A)$, and $(-1)^{\varepsilon(t)}$ equals $\big(\frac{-1}{t}\big)$. As both $\xi$ and $\upsilon$ are in $\{\pm1\}$, and Proposition \ref{char4notdivN} shows that the spaces $\mathcal{A}_{k+1/2}(\widetilde{\Gamma}_{Nt},\psi^{2k\pm1})$ are distinct (since $4 \nmid N$), this establishes the assertion by checking the two cases of $\varepsilon(t)$. It remains to consider the case where $t\equiv2(\mathrm{mod\ }4)$  and $4 \nmid N$. But the conductor of $\chi_{t}$ is divisible by 8 in this case, while $Nt$ is not. Hence if $A\in\widetilde{\Gamma}_{Nt}$ is as above and such that $d\equiv5(\mathrm{mod\ }8)$ and $\psi(A)=1$ then the fact that $\chi_{t}(d)=-1$ shows that $f_{t}$ cannot be in any of the spaces $\mathcal{A}_{k+1/2}(\widetilde{\Gamma}_{Nt},\psi^{2k+\upsilon})$ with $\upsilon\in\{\pm1\}$. The second assertion is thus also proved.

As for the characters, if $t$ is odd and $\big(\frac{-1}{t}\big)=\xi\upsilon$ then $d\mapsto\big(\frac{\xi\upsilon t}{d}\big)$ is defined modulo $t$. Otherwise it is defined modulo $8t_{2}$, and even $4t_{2}$ when $t=4t_{2}$, proving that $\eta$ is defined modulo $Nt$ under the conditions before. As $\chi_{t}$ is even, we find that $\eta(-1)$ is the product of $\chi(-1)=\xi$ and the image $\xi\upsilon$ of $-1$ under $d\mapsto\big(\frac{\xi\upsilon}{d}\big)$, which is indeed $\upsilon$. Now, if $f\in\mathcal{A}_{k+1/2}\big(\widetilde{\Gamma}_{0}(4N),\psi^{2k+1}\omega_{\chi}\big)$ then $f_{t}\big|_{k+1/2}A(\tau)=f\big|_{k+1/2}A_{t}(t\tau)$ (with $A\in\widetilde{\Gamma}_{0}(Nt)$ and $A_{t}\in\widetilde{\Gamma}_{0}(N)$ as above) equals $\psi^{2k+1}(A_{t})\omega_{\chi}(d)f_{t}(\tau)$, and we have seen that the coefficient here can be written as $\psi^{2k+1}(A)$ times the expression $\chi_{t}(d)\big(\frac{\xi}{d}\big)\chi(d)$ (where we have also used Equation \eqref{omegachi} and the condition $\chi(-1)=\xi$ for the definition of $\omega_{\chi}$). But the latter expression is just $\big(\frac{\upsilon}{d}\big)\eta(d)$ by the definition of $\eta$, and this is $\omega_{\eta}(d)$ since $\eta(-1)=\upsilon$. This proves that $f_{t}\big|_{k+1/2}A=\psi^{2k+1}(A)\omega_{\eta}(d)f_{t}$ for every such $A$, proving the third assertion as well.

This completes the proof of the lemma.
\end{proof}

We will now prove the following result.
\begin{thm}
Let $t$ be a positive square-free integer, and take $f\in\mathcal{M}_{k+1/2}(\widetilde{\Gamma}_{N}^{\psi})$. Assume that either $t$ is odd and $f\in\mathcal{M}_{k+1/2}^{+}(\widetilde{\Gamma}_{N},\psi^{\varepsilon})$ where
$\varepsilon=\big(\frac{-1}{t}\big)$, or that $t$ is even, $4|N$, $f$ contains only Fourier coefficients of even indices, and $\varepsilon\in\{\pm1\}$ is arbitrary. Then the function of $\sigma\in\mathcal{H}$ defined by \[\mathcal{S}_{t}^{(N)}f(\sigma):=-\sum_{\substack{d=1 \\ (d,N)=1}}^{N}\Bigg[\sum_{\substack{m=0 \\ (d+Nm,t)=1}}^{t-1}\bigg(\frac{\varepsilon t}{d+Nm}\bigg)\frac{\zeta_{Nt}^{(d+Nm)}(1-k)}{2}\Bigg]c_{0}^{(d)}+\sum_{l=1}^{\infty}\Bigg[\sum_{\substack{d|l \\ (d,Nt)=1}}d^{k-1}\bigg(\frac{\varepsilon t}{d}\bigg)c_{\frac{tl^{2}}{d^{2}}}^{(d)}\Bigg]\mathbf{e}(l\sigma)\] belongs to
$\mathcal{M}_{2k}\big(\Gamma_{1}^{\sqrt{1}}(N)\big)$. In case $f\in\mathcal{M}_{k+1/2}^{+}\big(\widetilde{\Gamma}_{0}(4N),\psi^{2k+1}\omega_{\chi}\big)$ for a Dirichlet character $\chi$ modulo $N$ satisfying
$\chi(-1)=\varepsilon(-1)^{k}$, we have
\[\mathcal{S}_{t}^{(N)}f(\sigma)=-\frac{L(1-k,\eta)}{2}c_{0}+\sum_{l=1}^{\infty}\Bigg[\sum_{\substack{d|l \\ (d,Nt)=1}}d^{k-1}\eta(d)c_{\frac{tl^{2}}{d^{2}}}\Bigg]\mathbf{e}(l\sigma),\]
with $\eta$ the character modulo $Nt$ which is defined by  $\eta(d):=\chi(d)\big(\frac{\varepsilon t}{d}\big)$ and satisfying $\eta(-1)=(-1)^{k}$. The modular form $\mathcal{S}_{t}^{(N)}f$ then lies in $\mathcal{M}_{2k}\big(\Gamma_{0}(N),\chi^{2}\big)$. In case the modular form $f$ is in $\mathcal{M}_{k+1/2}^{!}(\widetilde{\Gamma}_{N},\psi)$ or in
$\mathcal{M}_{k+1/2}^{!}\big(\widetilde{\Gamma}_{0}(4N),\psi^{2k+1}\omega_{\chi}\big)$, the same expansions produce meromorphic modular forms with poles of order $k$
at CM points. \label{Stf}
\end{thm}

\begin{rmk}
The simple relation
$\zeta_{N}^{(d)}=\sum_{m=1}^{t}\zeta_{Nt}^{(d+Nm)}$ and some conductor
considerations show that the coefficient of $c_{0}^{(d)}$ in Theorem \ref{Stf}
reduces to just $\big(\frac{\varepsilon t}{d}\big)\frac{\zeta_{N}^{(d)}(1-k)}{2}$, except in the case with $2|t$ and $N\equiv4(\mathrm{mod\ }8)$, where this coefficient equals $\big(\frac{\varepsilon t}{d}\big)\frac{\zeta_{2N}^{(d)}(1-k)-\zeta_{2N}^{(d+N)}(1-k)}{2}$. \label{parzetas}
\end{rmk}

\begin{proof}
The proof follows Section 3 of \cite{[N]}. Lemma \ref{relfft} and
a simple investigation of the Fourier coefficients show that the cases we
consider are the precisely the ones in which $f_{t}\in\mathcal{M}_{k+1/2}^{+}(\widetilde{\Gamma}_{Nt},\psi)$. By Theorem \ref{main},
\[\mathcal{S}_{1}^{(Nt)}f_{t}(\sigma)=-\sum_{\substack{h=1 \\ (h,Nt)=1}}^{Nt}\frac{\zeta_{Nt}^{(h)}(1-k)\tilde{c}_{0}^{(h)}}{2}+\sum_{l=1}^{\infty}\Bigg[\sum_{\substack{
d|l \\ (d,Nt)=1}}d^{k-1}\tilde{c}_{l^{2}/d^{2}}^{(d)}\Bigg]\mathbf{e}(l\sigma)\] yields an element of $\mathcal{M}_{2k}\big(\Gamma_{1}^{\sqrt{1}}(Nt)\big)$, where $\tilde{c}_{n}^{(d)}$ is the $n$\textsuperscript{th} Fourier coefficient of $\langle d \rangle_{+}f_{t}$. But Proposition \ref{diamSV} shows that we can write this function as $f_{t}\big|_{k+1/2,\psi}A$ for $A\in\widetilde{\Gamma}_{0}(16Nt)$ with lower right entry in $d+Nt\mathbb{Z}$. The proof of Lemma \ref{relfft} now shows that this function attains on $\tau\in\mathcal{H}$ the value $\overline{\psi}(A)f\big|_{k+1/2}A_{t}(t\tau)$ with $A_{t}\in\widetilde{\Gamma}_{0}(16N)$ having the same lower right entry as $A$, and that this coincide with $\big(\frac{t}{d}\big)f\big|_{k+1/2,\psi}A_{t}(t\tau)$. Replacing $\psi$ by $\psi^{\varepsilon}$ in the slash operator takes out a coefficient of $\big(\frac{\varepsilon}{t}\big)$. Using Proposition \ref{diamSV} once more, this slash operator is now just $\langle d \rangle_{\varepsilon}$ (or maybe its extended version), and since the variable is $t\tau$ we find that $\tilde{c}_{n}^{(d)}$ equals $\big(\frac{\varepsilon t}{d}\big)c_{n/t}^{(d)}$ in case $t|n$ and 0 otherwise. As $t$ is square-free and $(d,Nt)=1$, we find that
$t\big|\frac{l^{2}}{d^{2}}$ if and only if $t|l$. We thus replace $l$ by $tl$ in the expression for $\mathcal{S}_{1}^{(Nt)}f_{t}(\sigma)$, a function of $\sigma$ which transforms with respect to $\Gamma_{1}^{\sqrt{1}}(Nt)$, and see that this expression equals $\mathcal{S}_{t}^{(N)}f(t\sigma)=t^{-k}(\mathcal{S}_{t}^{(N)}f)\big|C_{t}(\sigma)$ (after the summation index change $h=d+Nm$). Here we have used the weight $2k$ slash operator associated with the matrix lying under the metaplectic element $C_{t}$ from Equation \eqref{ftCtdef}, for which we allow ourselves the slight abuse of notation and denote $C_{t}$ as well.

Now, since $f_{t}$ is modular with respect to $\widetilde{\Gamma}_{N}\cap\widetilde{\Gamma}_{0}(4Nt)$ and character $\psi( \cdot) \big(\frac{\varepsilon t}{\cdot}\big)$, the proof of Proposition \ref{LNphichi} shows that elements of $(\mathbb{Z}/Nt\mathbb{Z})^{\times}$ which are congruent to 1 modulo $N$ operate on $\mathcal{L}_{Nt}\mathcal{L}_{+}f_{t}$ through the character $\big(\frac{\varepsilon t}{\cdot}\big)$. The proof of Proposition \ref{mainchar} thus yields the modularity of $\mathcal{S}_{1}^{(Nt)}f_{t}$ with respect to $\Gamma^{\sqrt{1}}_{1}(N) \cap \Gamma_{0}(Nt)$. It follows that $\mathcal{S}_{t}^{(N)}f=t^{k}\mathcal{S}_{1}^{(Nt)}f_{t}\big|C_{t}^{-1}$ is modular with respect to the $C_{t}^{-1}$-conjugate  $\Gamma^{\sqrt{1}}_{1}(N)\cap\Gamma^{0}(t)$ of $\Gamma^{\sqrt{1}}_{1}(N) \cap \Gamma_{0}(Nt)$, and since $\mathcal{S}_{t}^{(N)}f$ is invariant under $T$ as well, the first assertion follows.

If $f\in\mathcal{M}_{k+1/2}\big(\widetilde{\Gamma}_{0}(4N),\psi^{2k+1}\omega_{\chi}\big)$ for such $\chi$, then $f_{t}\in\mathcal{M}_{k+1/2}\big(\widetilde{\Gamma}_{0}(4Nt),\psi^{2k+1}\omega_{\eta}\big)$ with the prescribed character $\eta$ by Lemma \ref{relfft} with $\xi=(-1)^{k}\varepsilon$ and $\upsilon=(-1)^{k}$. The second assertion thus follows by a similar argument, using Proposition \ref{mainchar}. Weakly holomorphic modular forms produce via $\mathcal{S}_{t}^{(N)}$ meromorphic modular forms as in Theorem \ref{main} and Proposition \ref{mainchar}. This completes the proof of the theorem.
\end{proof}

The assertion about the cuspidality of $\mathcal{S}_{1}^{(N)}f$ appearing in Remark \ref{cusptocusp} clearly extends to the map $\mathcal{S}_{t}^{(N)}f$ described in Theorem \ref{Stf}, as well as in the more general situations appearing in Propositions \ref{4divNlevN} and \ref{level} below. In the case where $f\in\mathcal{M}_{k+1/2}^{+}\big(\widetilde{\Gamma}_{0}(4N),\psi^{2k+1}\omega_{\chi}\big)$ and $\chi$ is quadratic, we expect the statement about Atkin--Lehner eigenfunctions from Remark \ref{ALprob} to extend to these situations as well.

Following Remark \ref{L+L-same} and Proposition \ref{diamSV}, we remark again that if $4 \nmid N$ (and $t$ is odd) then $\varepsilon$ is determined by the action of $\widetilde{\Gamma}$ on $f$, and the coefficients $c_{m}^{(d)}$ carry no ambiguity. On the other hand, we have
\begin{prop}
If $4 \nmid N$ then the assertions of Theorem \ref{Stf} hold for any $f\in\mathcal{M}_{k+1/2}^{!}(\widetilde{\Gamma}_{N}^{\psi})$, and are independent of the choice of $\varepsilon$. In particular, the level of $\mathcal{S}_{t}^{(N)}f$ is $N$ for every such $f$. \label{4divNlevN}
\end{prop}

\begin{proof}
Proposition \ref{diamSV} implies that the dependence of $c_{m}^{(d)}$ on $\varepsilon$ is the same as of $\big(\frac{\varepsilon}{d}\big)$. It follows that the product $\big(\frac{\varepsilon}{d}\big)c_{m}^{(d)}$ appearing in Theorem \ref{Stf} is independent of $\varepsilon$. Now, the assumption on $f$ appearing in that Theorem is that only Fourier coefficients whose indices are congruent to 0 or to t modulo 4 may not vanish, and the formula appearing there only uses these coefficients. But Proposition \ref{proj+4divN} provides us with a projection map $P_{t}$ (which is  $P_{\varepsilon}$ if $t$ is odd and $\big(\frac{-1}{t}\big)=\varepsilon$ and it is $P_{2}$ for even $t$) on $\mathcal{A}_{k+1/2}(\widetilde{\Gamma}_{N}^{\psi})$, whose action leaves only the required coefficients. Moreover, the conjugation formula for $T^{1/4}$ appearing in the proof of that Proposition shows that conjugating the matrix $A\in\widetilde{\Gamma}_{0}(16N)$ which we use in Proposition \ref{diamSV} to present the diamond operator $\langle d \rangle_{\varepsilon}$ gives another element of $\widetilde{\Gamma}_{0}(16N)$, with the same lower right entry modulo $4N$. It follows that $P_{t}$ commutes with the diamond operators, so that applying $\langle d \rangle_{\varepsilon}$ to $P_{t}f$ only takes the right coefficients from $\langle d \rangle_{\varepsilon}f$. It follows that $\mathcal{S}_{t}f$ is the same as $\mathcal{S}_{t}(P_{t}f)$, for which the assertions holds by Theorem \ref{Stf}. This proves the proposition.
\end{proof}

\smallskip

We conclude with discussing three questions, which turn out to be closely related. One is the dependence of the Shimura lift on the level of the space in which we consider $f$ (i.e., an element of $\mathcal{M}_{k+1/2}(\widetilde{\Gamma}_{N}^{\psi})$ lies also in $\mathcal{M}_{k+1/2}(\widetilde{\Gamma}_{MN}^{\psi})$ for any $M\in\mathbb{N}$). The second one is Shimura lifts with indices which are not square-free. The third question is the minimal level of the Shimura lift. We extend the definition of $\mathcal{S}_{t}$ so that $\mathcal{S}_{t}f$ is defined by the formula from Theorem \ref{Stf} to all $f\in\mathcal{M}_{k+1/2}(\widetilde{\Gamma}_{N},\psi^{\varepsilon})$ and all $t\in\mathbb{N}$ (not necessarily square-free), where we write the constant term of $\mathcal{S}_{t}f$ as $-\sum_{h=1,\ (h,Nt)=1}^{4Nt}\big(\frac{\varepsilon t}{h}\big)\frac{\zeta_{4Nt}^{(h)}(1-k)c_{0}^{(h)}}{2}$ for well-definedness in some cases (this gives the previous constant term wherever $h\mapsto\big(\frac{\varepsilon t}{h}\big)$ is defined modulo $Nt$ by the same argument as in Remark \ref{parzetas}). Propositions \ref{char4notdivN} and \ref{4divNlevN} imply that if $4 \nmid N$ then $\varepsilon$ is determined by $f$, while if $4|N$ the expression for $\mathcal{S}_{t}f$ is independent of the choice of $\varepsilon$. Given $s\in\mathbb{N}$, we recall the (Hecke type) operator $U_{s}$ sending a holomorphic Fourier expansion $\sum_{n=0}^{\infty}a_{n}\mathbf{e}(n\sigma)$ to the Fourier expansion $\sum_{n=0}^{\infty}a_{ns}\mathbf{e}(n\sigma)$.

Let now $N$, $t$, $M$, and $s$ be natural numbers, and denote the set of primes dividing $M$ but not $Nt$ by $I$. For any $J \subseteq I$ we define $p_{J}=\prod_{p \in J}p$, and $|J|$ is the cardinality of $J$.
\begin{prop}
The following assertions hold:
\begin{enumerate}[$(i)$]
\item For any $f\in\mathcal{M}_{k+1/2}^{!}(\widetilde{\Gamma}_{N},\psi^{\varepsilon})$ with $\varepsilon\in\{\pm1\}$, the equality \[\mathcal{S}_{t}^{(MN)}f=\sum_{J \subseteq I}\frac{(-1)^{|J|}}{p_{J}}\bigg(\frac{\varepsilon t}{p_{J}}\bigg)\cdot\mathcal{S}_{t}^{(N)}(\langle p_{J} \rangle_{\varepsilon}f)\big|C_{p_{J}}\] (with the slash operators having weight $2k$) holds. In case $\chi$ is a Dirichlet character modulo $N$ with the assumptions of that theorem and $f$ belongs to $\mathcal{M}_{k+1/2}^{!}\big(\widetilde{\Gamma}_{0}(4N),\psi^{2k+1}\omega_{\chi}\big)$, the term corresponding to $J$ is $\frac{(-1)^{|J|}}{p_{J}}\big(\frac{\varepsilon t}{p_{J}}\big)\chi(p_{J})\cdot\mathcal{S}_{t}^{(N)}f\big|C_{p_{J}}$.
\item If $s$ is odd then the conditions for $f_{ts^{2}}$ to be $\mathcal{A}_{k+1/2}^{+}(\widetilde{\Gamma}_{Nts^{2}},\psi)$ are the same as those for $f_{t}$ to be in  $\mathcal{A}_{k+1/2}^{+}(\widetilde{\Gamma}_{Nt},\psi)$. If $s$ is even then  $f_{ts^{2}}\in\mathcal{A}_{k+1/2}^{+}(\widetilde{\Gamma}_{Nts^{2}},\psi)$ for every $f\in\mathcal{A}_{k+1/2}(\widetilde{\Gamma}_{N}^{\psi})$.
\item For $f$ as in part $(i)$ we have the equality $\mathcal{S}_{ts^{2}}^{(N)}f=(\mathcal{S}_{t}^{(Ns)}f)\big|U_{s}$. If $t$ is square-free then the equality $\mathcal{S}_{1}^{(Nts)}f_{ts^{2}}=(st)^{-k}(\mathcal{S}_{t}^{(Ns)}f)\big|C_{st}$ also holds.
\end{enumerate} \label{highlev}
\end{prop}

\begin{proof}
Part $(i)$ is essentially Lemma 5 of \cite{[T]}, and is easily proved by comparing the co-primality condition on $d$ (or $d+Nm$, or $h$) in the definition with $N$ and with $MN$ and using the inclusion-exclusion principle. Part $(ii)$ easily follows from the definition of the plus-space, using the fact that $s^{2}\equiv1(\mathrm{mod\ }4)$ for odd $s$ while $4|s^{2}$ for even $s$. The first equality in part $(iii)$, which is roughly Lemma 6 of \cite{[T]}, follows immediately by comparing the Fourier coefficients of the functions involved. The second equality is established as in the proof of Theorem \ref{Stf}, noting that $ts^{2}$ divides $l^{2}$ if and only if $st|l$ when $t$ is square-free. This completes the proof of the proposition.
\end{proof}

Part $(i)$ of Proposition \ref{highlev} shows that if $N$, $f$, and $t$ satisfy the conditions of Theorem \ref{Stf} then $\mathcal{S}_{t}^{(MN)}f$ coincides with $\mathcal{S}_{t}^{(Np_{I})}f$, hence has level $Np_{I}$ (note that Proposition \ref{diamSV} shows that the diamond operators $\langle p_{J} \rangle_{\varepsilon}$ appearing there preserve the plus-space condition, so that the levels of the summands are known). A classical result, which is easy to prove, shows that in integral weights the operator $U_{s}$ takes modular forms of any level $L$ to modular forms of level $[L,s]$, the least common multiple $\frac{Ls}{(L,s)}$ of $L$ and $s$. This operator also commutes with the diamond operators, when defined using appropriate representatives. This combines with part $(iii)$ of Proposition \ref{highlev} to show that under our assumptions on $N$, $f$, and $t$, the level of $\mathcal{S}_{ts^{2}}^{(MN)}f$ is $[L,s]p_{J}$ where $J$ is the set of primes in $I$ which do not divide $s$ (note that while the level of $\mathcal{S}_{t}^{(Ns)}f$ may be larger than $N$, its least common multiple with $s$ coincides with $[N,s]$). Part $(iii)$ of Proposition \ref{highlev} suggests that using $f_{ts^{2}}$ rather than $f_{t}$ in the proof of Theorem \ref{Stf} may be useful, but part $(ii)$ shows that the only case where this idea can produce new results is where $N$, $f$, and $t$ do not satisfy the conditions of Theorem \ref{Stf} and $s=2$ (larger even values of $s$ will then follow from this case). Note that here, as well as in Theorem \ref{level} below, being of level $L$ means lying in $\mathcal{M}_{2k}\big(\Gamma_{0}(L)\cap\Gamma_{1}^{\sqrt{1}}(N)\big)$.

The accurate level of the Shimura lift in general is now determined as follows.
\begin{thm}
Let $f$ be a modular form of weight $k+\frac{1}{2}$ with respect to some congruence subgroup, and let $N$ be the minimal number such that $f\in\mathcal{M}_{k+1/2}(\widetilde{\Gamma}_{N}^{\psi})$. Given three positive integers $t$, $s$, and $M$ with $t$ square-free, we define $I$ as in Proposition \ref{highlev} and $J$ to be the set of primes in $I$ which do not divide $s$. Consider now the Shimura lift $\mathcal{S}_{ts^{2}}^{(MN)}f$, and we have the following cases:
\begin{enumerate}[$(i)$]
\item $t$ is odd and $f\in\mathcal{M}_{k+1/2}^{+}(\widetilde{\Gamma}_{N},\psi^{\varepsilon})$ with $\varepsilon=\big(\frac{-1}{t}\big)$.
\item $4|N$.
\item $Nt$ is odd and $M$ is even.
\item $4|s$.
\item $N$ is odd and $s$ is even.
\item $Ns$ is odd and $t$ is even.
\item $N\equiv2(\mathrm{mod\ }4)$, $4 \nmid s$, and we are not in case $(i)$.
\item $MNst$ is odd, $f$ lies in one of the subspaces $f\in\mathcal{M}_{k+1/2}(\widetilde{\Gamma}_{N},\psi^{\varepsilon})$ from Proposition \ref{char4notdivN}, but if $\varepsilon=\big(\frac{-1}{t}\big)$ then $f$ is not in the plus-space.
\end{enumerate}
In cases $(i)$--$(v)$ the function $\mathcal{S}_{ts^{2}}^{(MN)}f$ has level $p_{J}[N,s]$. In the two cases $(vi)$ and $(vii)$ its level is $2p_{J}[N,s]$. In the remaining case $(viii)$ we find that the combination $\mathcal{S}_{ts^{2}}^{(MN)}f-\frac{1}{2}\big(\frac{2}{t}\big)\cdot\mathcal{S}_{ts^{2}}^{(MN)}(\langle2\rangle_{\varepsilon}f)\big|C_{2}$ has level $2p_{J}[N,s]$. The assertions about characters, and about meromorphicity of lifts of weakly holomorphic modular forms, extend to this case. \label{level}
\end{thm}

\begin{proof}
The assertion in cases $(i)$ and $(ii)$ was proved just before the statement of the theorem. We henceforth assume we are in neither of these cases. Consider the case with $M=s=1$. We plug the equality $\mathcal{S}_{1}^{(2Nt)}f_{4t}=(2t)^{-k}(\mathcal{S}_{t}^{(2N)}f)\big|C_{2t}$ from part $(iii)$ of Proposition \ref{highlev} and the fact that $f_{4t}\in\mathcal{A}_{k+1/2}^{+}(\widetilde{\Gamma}_{4Nt},\psi)$ by part $(ii)$ of that proposition into the proof of Theorem \ref{Stf}. The argument now shows that $\mathcal{S}_{t}^{(2N)}f$ has level $2N$. For general $M$ and $s$ we now apply the same argument from above, and observe that $\mathcal{S}_{ts^{2}}^{(2MN)}f$ coincides with $\mathcal{S}_{ts^{2}}^{(MN)}f$ except in case $(viii)$ where the former Shimura lift is the asserted combination (part $(i)$ of Proposition \ref{highlev} again). The remaining assertions now follow by comparing $[2Np_{I,2},s]$, where $p_{I,2}$ is the odd part of $p_{I}$, with $[N,s]p_{J}$ in all these cases. This proves the theorem.
\end{proof}
Note that Theorem \ref{level} is almost exhaustive, in the sense that if we are not in any of the cases $(i)$--$(v)$ then the conditions of precisely one of the cases $(vi)$, $(vii)$, or $(viii)$ must hold, except maybe the assertion that $f$ lies in one of the subspaces $\mathcal{M}_{k+1/2}(\widetilde{\Gamma}_{N},\psi^{\varepsilon})$ in the latter case. Without this assumption one has to decompose $f$ into its two $\mathcal{M}_{k+1/2}(\widetilde{\Gamma}_{N},\psi^{\varepsilon})$ components, and the correction term in the combination is different for each component (because of the index of the diamond operator $\langle2\rangle_{\varepsilon}$). Unfortunately, our proof of Theorem \ref{level} does not show (at least not directly) the modularity of $\mathcal{S}_{ts^{2}}^{(MN)}f$ itself in case $(viii)$. Note that the Dirichlet $L$-function appearing in the full $L$-function considered in \cite{[Sh]} and \cite{[N]} already excludes the 2-adic multiplier, showing that these references do not treat this case as well.

\noindent\textsc{Fachbereich Mathematik, AG 5, Technische Universit\"{a}t
Darmstadt, Schlossgartenstrasse 7, D-64289, Darmstadt, Germany}

\noindent E-mail address: li@mathematik.tu-darmstadt.de

\medskip

\noindent\textsc{Einstein Institute of Mathematics, the Hebrew University of Jerusalem, Edmund Safra Campus, Jerusalem 91904, Israel}

\noindent E-mail address: zemels@math.huji.ac.il


\begin{thebibliography}{}{}

\bibitem[B]{[B]} Borcherds, R. E., \textsc{Automorphic Forms with Singularities on Grassmannians}, Invent. Math., vol. 132, 491--562 (1998).
\bibitem[Br1]{[Br1]} Bruinier, J. H., \textsc{Borcherds Products on $O(2,l)$ and Chern Classes of Heegner Divisors}, Lecture Notes in Mathematics 1780, Springer--Verlag (2002).
\bibitem[Br2]{[Br2]} Bruinier, J. H., \textsc{On the Converse Theorem for Borcherds Products}, J. Algebra, vol 397, 315--342 (2014).
\bibitem[BF]{[BF]} Bruinier, J. H., Funke, J., \textsc{On Two Geometric Theta Lifts}, Duke Math J., vol 125 no. 1, 45--90 (2004).
\bibitem[BO]{[BO]} Bruinier, J. H., Ono, K., \textsc{Heeger Divisors, $L$-Functions, and Harmonic Weak Maass Forms}, Ann. of Math., vol 172, 2135--2181 (2010).
\bibitem[C1]{[C1]} Cipra, B., \textsc{On the Niwa--Shintani Theta-Kernel Lifting of Modular Forms}, Nagoya Math J., vol 91, 49--117 (1983).
\bibitem[C2]{[C2]} Cipra, B., \textsc{On the Shimura Lift, apr\`{e}s Selberg}, J. Number Theory, vol 32, 58--64 (1989).
\bibitem[DS]{[DS]} Diamond, F., Shurman, J., \textsc{A First Course in Modular Forms}, GTM 228, Springer-Verlag, New York (2005).
\bibitem[EZ]{[EZ]} Eichler, M., Zagier, D., \textsc{The Theory of Jacobi Forms}, Progress in Mathematics, vol. 55, Birkh\"{a}user, Boston, Basel, Stuttgart, (1985).
\bibitem[EMOT]{[EMOT]} Erd\'{e}lyi, A, Magnus, W., Oberhettinger, F., Tricomi, F., \textsc{Higher Transcendetal Functions}, McGraw--Hill (1953).
\bibitem[HM]{[HM]} Harvey, J., Moore, G., \textsc{Algebras, BPS states, and strings}, Nuclear Phys. B, vol 463 no. 2-3, 315--368 (1996).
\bibitem[HN]{[HN]} Hansen, D., Naqvi, Y., \textsc{Shimura Lifts of Half-Integral Weight Modular Forms Arising from Theta Functions}, Ramanujan J., vol 17 issue 3, 343--354 (2008).
\bibitem[K1]{[K1]} Kohnen, W., \textsc{Modular Forms of Half-Integral Weight on $\Gamma_{0}(4)$}, Math. Ann., vol 248, 249--266 (1980).
\bibitem[K2]{[K2]} Kohnen, W., \textsc{Newforms of Half-Integral Weight}, J. Reine Angew. Math., vol 333, 32--72 (1982).
\bibitem[N]{[N]} Niwa, S., \textsc{Modular Forms of Half Integral Weight and the Integral of Certain Theta-Functions}, Nagoya Math J., vol 56, 147--161 (1974).
\bibitem[P]{[P]} Purkait, S., \textsc{On Shimura's decomposition}, Int. J. Number Theory 9, no. 6, 1431--1445 (2013).
\bibitem[Sch]{[Sch]} Scheithauer, N. R., \textsc{The Weil representation of $SL_{2}(\mathbb{Z})$ and some applications}, Int. Math. Res. Not., no. 8, 1488--1545 (2009).
\bibitem[Sh]{[Sh]} Shimura, G., \textsc{On Modular Forms of Half Integral Weight}, Ann. of Math., vol. 97, 440--481 (1973).
\bibitem[Sn]{[Sn]} Shintani, T., \textsc{On the construction of holomorphic cusp forms of half-integral weight}, Nagoya Math. J., 58, 83--126 (1975).
\bibitem[Str]{[Str]} Str\"{o}mberg, F., \textsc{Weil Representations Associated to Finite Quadratic Modules}, Math. Z., vol 275 issue 1, 509--527 (2013).
\bibitem[T]{[T]} Tsuyumine, S., \textsc{On Shimura Lifting of Modular Forms}, Tsukuba J. Math., vol 23 no. 3, 465--483 (1999).
\bibitem[UY]{[UY]} Ueda, M.; Yamana, S., \textsc{On Newforms for Kohnen Plus Spaces}, Math. Z., vol 264, no. 1, 1--13 (2010).
\bibitem[Ze1]{[Ze1]} Zemel, S., \textsc{A $p$-adic Approach to the Weil Representation of Discriminant Forms Arising from Even Lattices}, Math. Ann. Qu\'{e}bec, vol 39 issue 1, 61--89 (2015).
\bibitem[Ze2]{[Ze2]} Zemel, S., \textsc{A Gross--Kohnen--Zagier Type Theorem for Higher-Codimensional Heegner Cycles}, Research in Number Theory 1 (2015), Art.\ 23.
\bibitem[Ze3]{[Ze3]} Zemel, S., \textsc{Normalizers of Congruence Groups in $SL_{2}(\mathbb{R})$ and Automorphisms of Lattices}, Int. J. Number Theory 13, no. 5, 1275--1300 (2017).

\end{thebibliography}
\end{document}